\documentclass[11pt,reqno]{amsart}

\usepackage[margin=1.15in]{geometry}
\usepackage{amsmath,amssymb,amsthm,mathtools,mathrsfs}
\usepackage{bbm}
\usepackage{enumerate}
\usepackage{cite}
\usepackage[dvipsnames]{xcolor}
\usepackage[colorlinks=true,linkcolor=blue,citecolor=blue,urlcolor=blue]{hyperref}
\usepackage[nameinlink,noabbrev]{cleveref}
\usepackage{microtype}

\numberwithin{equation}{section}

\newtheorem{theorem}{Theorem}[section]
\newtheorem{proposition}[theorem]{Proposition}
\newtheorem{lemma}[theorem]{Lemma}
\newtheorem{corollary}[theorem]{Corollary}
\newtheorem{remark}[theorem]{Remark}
\newtheorem{definition}[theorem]{Definition}

\newcommand{\eps}{\varepsilon}
\newcommand{\R}{\mathbb R}
\newcommand{\E}{\mathbb E}

\newcommand{\Law}{\operatorname{Law}}
\newcommand{\Tr}{\operatorname{Tr}}
\newcommand{\Id}{\mathbb I}

\newcommand{\cN}{\mathcal N}
\newcommand{\dif}{\mathrm d}

\def\a{\alpha}
\def\b{\beta}
\def\p{\partial}
\def\d{\delta}
\def\g{\gamma}

\def\no{\nonumber}
\def\cA{{\mathcal A}}

\def\cD{{\mathcal D}}

\def\cH{{\mathcal H}}

\def\cP{{\mathcal P}}
\def\cQ{{\mathcal Q}}
\def\cR{{\mathcal R}}
\def\cS{{\mathcal S}}

\def\cU{{\mathcal U}}
\def\cV{{\mathcal V}}

\def\mI{{\mathbb I}}

\def\sI{{\mathscr I}}

\def\sL{{\mathscr L}}

\def\sR{{\mathscr R}}
\def\sS{{\mathscr S}}

\def\sU{{\mathscr U}}

\def\bt{\begin{theorem}}
\def\et{\end{theorem}}
\def\bl{\begin{lemma}}
\def\el{\end{lemma}}
\def\br{\begin{remark}}
\def\er{\end{remark}}
\def\bd{\begin{definition}}
\def\ed{\end{definition}}
\def\bp{\begin{proposition}}
\def\ep{\end{proposition}}
\def\bc{\begin{corollary}}
\def\ec{\end{corollary}}

\newcommand{\ip}[2]{\left\langle #1,#2\right\rangle}
\newcommand{\norm}[1]{\left|#1\right|}

\allowdisplaybreaks

\begin{document}

\title{Central limit theorem for the Smoluchowski--Kramers approximation}

\date{}

\author{Yun Li, Longjie Xie$^\dag$, Li Yang and Qianning Zhang}
\address{Yun Li:
School of Mathematics and Statistics, Jiangsu Normal University,
Xuzhou, Jiangsu 221000, P.R.China\\
Email: liyun@jsnu.edu.cn
}

	\address{Longjie Xie:
		School of Mathematics and Statistics, Jiangsu Normal University,
		Xuzhou, Jiangsu 221000, P.R. China\\
		Email: longjiexie@jsnu.edu.cn  }

\address{Li Yang:
	School of Mathematical Sciences, Tiangong University,
	Tianjin 300387, P.R.China\\
	Email: lyang@tiangong.edu.cn
}

	\address{Qianning Zhang:
		School of Mathematics and Statistics, Jiangsu Normal University,
		Xuzhou, Jiangsu 221000, P.R. China\\
		Email: zhangqianning@jsnu.edu.cn  }


\thanks{This work is supported by National Key R\&D program of China (No. 2023YFA1010103) and NSFC (No. 12301179, 12401580).
}

\begin{abstract}
We establish the central limit theorem (CLT) for Smoluchowski--Kramers
approximation of Langevin equations with state-dependent,
possibly non-symmetric friction and multiplicative noise. To the best of our knowledge, this is the  first such result, even in the case of constant friction and additive noise.
The result differs in an essential way from  classical  averaging CLTs
and reveals a   new two-scale fluctuation phenomenon.
 The normalized difference process
retains a persistent fast component and is generally  not
tight in \(C([0,T];\R^d)\), thus  the usual continuous-path CLT cannot hold.
We derive an exact  decomposition of the
normalized difference into an accumulated slow fluctuation and a rapidly relaxing
velocity correction. The auxiliary slow fluctuation converges to a
linear diffusion driven by both the original Brownian motion and   an
additional independent Brownian noise whose effective covariance is characterized explicitly through  a quadratic Poisson corrector. For every fixed \(t>0\), the fast correction converges conditionally to
its local Gaussian equilibrium. Consequently, the  fixed-time fluctuation limit  is identified as a Gaussian mixture:
\[
\frac{\bar X_t-X_t^\varepsilon}{\sqrt{\varepsilon}}
\Longrightarrow
\bar Y_t+N^{1/2}(\bar X_t)\xi,
\]
where $\bar X_t$ is the overdamped limit, \(\bar Y_t\) is the limiting slow fluctuation, \(N\) is the covariance of the terminal  velocity correction and solves a Lyapunov equation, and \(\xi\) is a standard
Gaussian random variable independent of \((\bar X_t,\bar Y_t)\).
A quantitative  convergence rate is also provided.
The result fills the remaining gap in the asymptotic theory of the overdamped approximation at the central limit scale, complementing the existing law of large numbers, moderate deviation and large deviation results.

		\bigskip
		
		

\end{abstract}

\maketitle

\section{Introduction}

\subsection{Background and the  problem}

The Smoluchowski--Kramers (S--K) approximation is a fundamental reduction
principle for stochastic dynamics with inertia:
when the mass of a particle is small, its velocity relaxes on a
much shorter time scale than its position, and the position is
approximated by an overdamped diffusion.
This reduction underlies many models of Brownian motion and
provides a tractable description of systems whose inertial
dynamics would otherwise require a phase-space formulation. We refer to \cite{F,PS,FW,CFAMZ} for a comprehensive overview.

\smallskip
Consider the following stochastic Langevin equation in $\R^d$ (\(d\geq1\)):
\begin{equation}\label{eq:Langevin}
\eps \ddot X_t^\eps
=
F(X_t^\eps)
-\gamma(X_t^\eps)\dot X_t^\eps
+\sigma(X_t^\eps)\dot W_t,
\qquad
X_0^\eps=x,\quad \dot X_0^\eps=v,
\end{equation}
where \(F:\R^d\to\R^d\) is an external force,
\(\gamma:\R^d\to\R^{d\times d}\) is a possibly non-symmetric friction
matrix,  \(\sigma:\R^d\to\R^{d\times d}\) is the noise coefficient, and \(W=(W_t)_{t\geq0}\) is a standard
\(d\)-dimensional Brownian motion. The parameter $\eps>0$ represents the mass of the particle.
Under suitable assumptions, the position
\(X_t^\varepsilon\) converges, as \(\varepsilon\to0\), to the
solution of the overdamped  equation
\begin{equation}\label{eq:overdamped}
\dif\bar X_t
=
B(\bar X_t)\dif t+a(\bar X_t)\dif W_t,
\qquad
\bar X_0=x,
\end{equation}
where
\begin{equation}\label{eq:defBa}
a(x):=\gamma^{-1}(x)\sigma(x),
\qquad
B(x):=\gamma^{-1}(x)F(x)+\cS(x).
\end{equation}
The additional vector field \(\cS\) is known as the noise-induced drift whose components are given by
\begin{equation}\label{eq:def-S}
\cS_i(x)
=
\sum_{j,k=1}^d
\partial_{x_j}\big(\gamma^{-1}\big)_{ik}(x)M_{kj}(x),
\qquad i=1,\ldots,d,
\end{equation}
where \(M(x)\) is the unique symmetric positive definite solution of
the Lyapunov equation
\begin{equation}\label{eq:lyapM}
\gamma(x)M(x)+M(x)\gamma^*(x)
=
\sigma(x)\sigma^*(x).
\end{equation}
More precisely, for any $T>0$ and $q\geq 1$, we have
\begin{equation}\label{eq:skrate}
\sup_{0\leq t\leq T}
\E\norm{X_t^\eps-\bar X_t}^q
\leq C_{q,T}\,\eps^{q/2}.
\end{equation}
Thus, the overdamped equation gives the functional law of large
numbers (LLN). This first-order approximation has been studied in a variety of finite and infinite dimensional settings; see, among others,
\cite{BC25,CF,HAN24,HMVW,HVW,SAL19,CFAM,BW20}.  Note that the noise-induced drift in (\ref{eq:def-S})
vanishes  for constant friction. The Poisson equation approach connects these results to
diffusion approximation for fully coupled multiscale systems;
see, e.g., \cite{HP2,PV,PS,RX,KM}. Extensions to irregular
coefficients or vanishing friction \cite{XY,FHW13}, Hamiltonian systems \cite{BW18,BW}, and geometric settings \cite{BHVW17,FGL15} further
demonstrate the robustness of the first-order approximation.

\smallskip
The first-order theory determines the effective motion of the position. Meanwhile,   the fluctuation theory has been developed extensively at the moderate and large deviation scales; see, e.g., \cite{CLW,MS,QY22} for the moderate deviation  principle (MDP) and \cite{CS,HS17,S13} for the large deviation  principle (LDP) of the S--K approximation. However, the central limit theorem (CLT) at the typical  fluctuation scale \(\sqrt{\eps}\) has remained unknown, even when both friction and noise coefficient are state independent. Define
the normalized difference process
\begin{equation}\label{eq:defZ}
Z_t^\eps
:=
\frac{\bar X_t-X_t^\eps}{\sqrt{\eps}}.
\end{equation}
Identifying the limit of $Z_t^\eps$ is
the pivotal level in the hierarchy
\[
\text{LLN}\ \longrightarrow\ \text{CLT}\ \longrightarrow\ \text{MDP}\ \longrightarrow\ \text{LDP}.
\]
It  quantifies the leading stochastic error made by the
overdamped approximation and  provides the local quadratic information underlying both moderate and large deviation rate functions.

\smallskip
This problem differs in an essential way  from the usual CLT problem
for averaging (see, e.g., \cite{Ce,HP,LPSY,RX2021,DL,DLPV}). A direct analogy with normal deviations for averaging would suggest
that \(Z_t^\varepsilon\) converges to a diffusion obtained by linearizing
\eqref{eq:overdamped}. This is not correct.  The main obstacle is that the normalized difference
\(
Z_t^\varepsilon
\)
not only contains an accumulated position
correction, but  also retains a non-vanishing fast velocity fluctuation (see our decomposition (\ref{eq:exact-decomposition-intro}) below).
This fast
component has order-one amplitude and correlation time of order
\(\varepsilon\),
it disappears under   macroscopic time
averaging but survives at every fixed positive time.
Thus \(Z^\varepsilon\) is generally not tight in
\(C([0,T];\mathbb R^d)\), and a standard continuous-path CLT fails.
The correct second-order description must therefore account for both the accumulated   fluctuation of the position dynamics and the local equilibrium of the rapidly relaxing velocity. The purpose of this paper is to identify these two contributions, their interaction through the common noise, and the resulting fixed-time limiting distribution.

\smallskip
Our result completes the overdamped limit theory at the typical fluctuation level and reveals  a  two-scale
fluctuation  mechanism that distinguishes the
S--K approximation from classical averaging: in the small mass limit, the fast thermal motion is not completely averaged out at the level of typical fluctuations, it leaves a residual Gaussian imprint whose   center is shifted by the slow fluctuation. Thus the S--K approximation, while exact at the LLN level, acquires a two-component structure at the fluctuation scale.

\subsection{Exact fluctuation decomposition and the limiting distribution}

Introduce the rescaled velocity
\begin{equation*}
V_t^\eps:=\sqrt{\eps}\,\dot X_t^\eps.
\end{equation*}
Then equation \eqref{eq:Langevin} is   equivalent to
\begin{equation}\label{eq:xv}
\left\{
\begin{aligned}
\dif X_t^\eps
&=\frac1{\sqrt{\eps}}V_t^\eps\,\dif t,
\\
\dif V_t^\eps
&=\frac1{\sqrt{\eps}}F(X_t^\eps)\,\dif t
-\frac1\eps\gamma(X_t^\eps)V_t^\eps\,\dif t
+\frac1{\sqrt{\eps}}\sigma(X_t^\eps)\,\dif W_t,
\end{aligned}
\right.
\end{equation}
with $X_0^\eps=x$ and $V_0^\eps=\sqrt{\eps}v$.
For a frozen position $x$, the leading velocity dynamics of (\ref{eq:xv}) is an
Ornstein--Uhlenbeck process satisfying
\begin{equation}\label{eq:frozenv}
\dif V_t^x=-\gamma(x)V_t^x\dif t+\sigma(x)\dif W_t.
\end{equation}
Its unique invariant measure is given by
\[
\zeta^x=\cN(0,M(x)),
\]
where \(M(x)\) is determined by \eqref{eq:lyapM}.
The process \(V_t^\varepsilon\) fluctuates on the time scale $\eps$ with
amplitude of order one.
Hence the corresponding velocity correction to the position has
size $\sqrt{\eps}$ and  remains visible in
\eqref{eq:defZ}.

\smallskip
The persistent part of the position error is captured by the linear
Poisson corrector
\[
\Gamma(x,v):=\gamma^{-1}(x)v.
\]
Let $\mathscr L_0(x)$ be the generator of the frozen velocity $V_t^x$, i.e.,
\begin{equation}\label{eq:L0}
\mathscr L_0(x)
=
-\langle\gamma(x)v,\nabla_v\rangle
+\frac12\sigma(x)\sigma^*(x):\nabla_{v}^2.
\end{equation}
Then we have
\[
\mathscr L_0(x)\Gamma(x,v)=-v.
\]
Applying It\^o's formula to
\(\Gamma(X_t^\varepsilon,V_t^\varepsilon)\) yields the exact
decomposition (see Lemma \ref{lem:decz} for the detailed proof)
\begin{equation}\label{eq:exact-decomposition-intro}
Z_t^\varepsilon
=
Y_t^\varepsilon
+\Gamma(X_t^\varepsilon,V_t^\varepsilon)
-\sqrt{\varepsilon}\,\Gamma(x,v),
\end{equation}
where $Y_t^\eps$ is the relative slow fluctuation part given by
\begin{align}\label{eq:defY}
Y_t^\eps&=\frac{1}{\sqrt{\eps}}
\int_0^t\big[B(\bar X_s)-B(X_s^\eps)\big]\dif s
+\frac{1}{\sqrt{\eps}}
\int_0^th(X_s^\eps,V_s^\eps)\dif s\no\\
&\quad+\frac{1}{\sqrt{\eps}}\int_0^t\big[a(\bar X_s)-a(X_s^\eps)\big]\dif W_s.
\end{align}
Here, $B$ and $a$ are defined by \eqref{eq:defBa}, and
for \(i=1,\dots,d\),
\begin{equation}\label{eq:defh}
h_i(x,v):=\cS_i(x)-\sum_{j,k=1}^d\partial_{x_j}(\gamma^{-1})_{ik}(x)v_kv_j.
\end{equation}
Identity \eqref{eq:exact-decomposition-intro} is the  starting point of
our analysis. The process \(Y_t^\varepsilon\) evolves
on the macroscopic time scale, whereas the term
\(\Gamma(X_t^\varepsilon,V_t^\varepsilon)\) remains of order one and
relaxes on the fast time scale. The decomposition isolates this term
before taking a limit and thereby separates the accumulated slow
fluctuation from the instantaneous local equilibrium.

\smallskip
We shall prove that  \(Y_t^\eps\) converges to the solution
of the linear SDE
\begin{equation}\label{eq:limitY}
\dif\bar Y_t=\nabla B(\bar X_t)\bar Y_t\,\dif t+\nabla a(\bar X_t)[\bar Y_t]\,\dif W_t+\sqrt{\cQ(\bar X_t)}\dif\tilde W_t,\quad \bar Y_0=0,
\end{equation}
where \(\tilde W\) is a \(d\)-dimensional Brownian motion
independent of \(W\).
This equation describes the accumulated position
fluctuation. The first two terms in \eqref{eq:limitY} are the linearization of the drift and diffusion coefficients of
the
overdamped equation along \(\bar X_t\), while the additional independent Brownian noise
records the  effect of centered velocity fluctuations.
The covariance \(\cQ\)  of this additional noise is not obtained by simply averaging separate
variance contributions. Indeed, two fast martingales exist: one arises from the linearization of the
common-noise coefficient, and the other from a quadratic Poisson
corrector for the centered velocity drift. Their cross-variation must
be retained before averaging, see  formula (\ref{eq:defQ}) below for the precise definition.
The fast term in
\eqref{eq:exact-decomposition-intro}  has a different
limiting behavior. We shall show that, at every fixed \(t>0\), it equilibrates
conditionally on the slow state and converges to a centered Gaussian
variable with covariance
\[
N(x)
:=
\gamma^{-1}(x)M(x)\gamma^{-*}(x).
\]

\smallskip
By combining the slow diffusion approximation with a conditional local-equilibrium estimate, we obtain an exact characterization of the fixed-time fluctuation limit, namely,
\[
Z_t^\eps
\Longrightarrow
Z_t:=\bar Y_t+N^{1/2}(\bar X_t)\xi,
\]
where \(\xi\sim\cN(0,\Id_d)\) is independent of
\((\bar X_t,\bar Y_t)\).  Equivalently, conditionally on the limiting slow state,
\[
\operatorname{Law}
\bigl(Z_t\mid\bar X_t,\bar Y_t\bigr)
=
\mathcal N\bigl(\bar Y_t,N(\bar X_t)\bigr).
\]
The limit is therefore a Gaussian
mixture.
Under \(C^{1+\beta}\)-type regularity  assumptions on the coefficients,  we obtain that for every $t\in(0,T]$,
\[
\begin{aligned}
&\left|
\E\varphi(Z_t^\eps)
-
\E\left(\int_{\R^d}
\varphi(z)
\,\mu^{\bar X_t,\bar Y_t}(\dif z)\right)
\right|\leq
C_T\,\eps^{\beta/2}
+
C_0\,\mathrm e^{-\kappa t/\eps},
\end{aligned}
\]
where for every fixed $(x,y)$,
\[
\mu^{x,y}:=\cN(y,N(x)).
\]
The algebraic part $\eps^{\beta/2}$ measures the accuracy of the fluctuation
approximation under the stated coefficient regularity. The exponential term is an unavoidable initial layer: at \(t=0\),
the fast fluctuation starts from a deterministic state and has not yet
reached its local equilibrium. Relaxation occurs on the scale $\eps$, and  reaching the algebraic
accuracy uniformly requires times of order
$\eps\log(1/\eps)$.

\smallskip
The result   has several   consequences for the fluctuation theory
of the overdamped approximation.

\vspace{1mm}
\begin{enumerate}[{\it (i)}]
\item It identifies the correct CLT for the S--K approximation. As far as we know, this is the first such result even for   constant friction and additive noise. It completes the limit theory of small mass Langevin dynamics, which previously consisted of a LLN, moderate deviations  and large deviations. Moreover, it has direct implications for thermodynamic fluctuations in heterogeneous media, exit-time statistics, and   rare event analysis, see, e.g., \cite{CS,F}.

\smallskip
\item  It shows that the S--K  CLT for Langevin dynamics  is fundamentally different from the classical  averaging CLT for   fast-slow systems.

\begin{itemize}
\item First, in a standard averaging problem, the normalized slow  difference often converges on path-space to an effective diffusion obtained by linearizing the averaged equation. Here the physical position difference contains an additional fast term that disappears under macroscopic time averaging but remains observable at each fixed time. Thus no continuous-path CLT can
hold.
This is not a technical defect of the proof but an intrinsic feature of the small mass fluctuation problem.

\vspace{1mm}
    \item Second, the overdamped limit $\bar X_t$  alone does not contain enough information to describe the typical fluctuations, one must also track the auxiliary process $Y_t^\eps$ introduced in (\ref{eq:defY}). Its limit
\(\bar Y_t\)  encodes how state-dependent friction, multiplicative noise, and the noise-induced drift jointly propagate slow fluctuations. In particular, the terminal limit is a  coupled diffusion pair \((\bar X_t,\bar Y_t)\) with a conditionally Gaussian correction. This is the fluctuation-level counterpart of the well-known  noise-induced drift phenomenon: at the LLN level, the state dependence of $\gamma$ and $\sigma$ modifies the effective drift through $\cS$; at the CLT level, the same state dependence modifies both the center and the covariance of the fluctuation distribution.
\end{itemize}

\smallskip
\item It provides a quantitative  second-order reduction of the inertial dynamics. The decomposition \eqref{eq:exact-decomposition-intro}, together with the limiting description, yields the second-order approximation
\[
X_t^\eps\;\approx\;\bar X_t-\sqrt{\eps}\,\big(\bar Y_t+N^{1/2}(\bar X_t)\xi\big),
\]
in which the linear fluctuation equation and the local velocity equilibrium play complementary roles. The error estimate of order \(\varepsilon^{\beta/2}\), together with the exponentially decaying initial-layer term and the explicit formulas for the covariances \(\cQ\) and $N$, shows that the two-scale fluctuation structure emerges on a time scale of order \(\varepsilon\) and is stable and computable. This is relevant for uncertainty quantification and for the design of second-order numerical  schemes for  Langevin  systems.

\end{enumerate}

\subsection{Notation and organization}
Throughout this paper,
for \(m\in\R^d\) and a symmetric positive semidefinite matrix
\(\Sigma\in\R^{d\times d}\), we use \(\cN(m,\Sigma)\) to denote  the
\(d\)-dimensional Gaussian distribution with mean \(m\) and covariance
\(\Sigma\). In particular, \(\cN(0,\Id_d)\) is the standard Gaussian
distribution in \(\R^d\), and we write \(X\sim\cN(m,\Sigma)\) if \(X\) has this
distribution.
Weak convergence of random variables is
written as
\(
X_n\Rightarrow X.
\)

The transpose of a matrix \(A\) is denoted by \(A^*\), and
\(A^{-*}:=(A^{-1})^*\).
For matrices \(A\) and \(B\) of the same size, we write
\[
A:B:=\operatorname{Tr}(AB^*).
\]
For simplicity, the Euclidean norm and the corresponding matrix norm are both denoted
by \(|\cdot|\). For a matrix-valued function \(a\),
\(
\nabla a(x)[y]
\)
denotes its directional derivative at \(x\) in the direction \(y\), i.e.,
\begin{equation}\label{eq:day}
\big(\nabla a(x)[y]\big)_{ij}=\sum_{k=1}^d\p_{x_k}a_{ij}(x)y_k.
\end{equation}

\smallskip
Given a function space, the subscript \(b\) will stand for boundedness, while
the subscript \(p\) stands for polynomial growth in \(y\) and \(v\). More precisely, for a
function \(f(t,z,x,y,v) \in L^\infty_p:= L^\infty_p\bigl(\R_+ \times\R^{4d}\bigr)\),
we mean that there exist constants $C,p>0$ such that
\[
\|f(\cdot,\cdot,\cdot,y,v)\|_\infty:=\sup_{t\in\R_+}\sup_{z,x\in\R^d}|f(t,z,x,y,v)|
\leq C\big(1+\norm{y}^p+\norm{v}^p\big).
\]
For \(0<\a_1,\a_2,\a_3,\a_4,\a_5\leq1\), the space \(C_p^{\a_1,\a_2,\a_3,\a_4,\a_5}:=C^{\a_1,\a_2,\a_3,\a_4,\a_5}_{t,z,x,y,v}(\R_+ \times\R^{4d})\) consists of all functions \(f\) satisfying the indicated H\"{o}lder regularity separately in each  variable, with polynomial growth in \(y\) and \(v\). More precisely, there exist constants $C,p>0$ such that
\begin{align*}
&|f(t,z,x_1,y,v)-f(t,z,x_2,y,v)|
\leq C\big(1+\norm{y}^p+\norm{v}^p\big)|x_1-x_2|^{\a_3}, \\
&|f(t,z,x,y_1,v)-f(t,z,x,y_2,v)|
\leq C\big(1+\norm{y_1}^p+\norm{y_2}^p+\norm{v}^p\big)|y_1-y_2|^{\a_4},
\end{align*}
and the analogous conditions hold separately in the variables \(t,z,v\).
For \(\a_1,\a_2,\a_3,\a_4,\a_5>1\), the space \(C_p^{\a_1,\a_2,\a_3,\a_4,\a_5}:=C^{\a_1,\a_2,\a_3,\a_4,\a_5}
_{t,z,x,y,v}(\R_+ \times\R^{4d})\) consists of all functions
\(f\) satisfying the indicated anisotropic derivatives with polynomial growth in \(y\) and \(v\), and
\[
\p_t^{k_1}\nabla_z^{k_2}\nabla_x^{k_3}
\nabla_y^{k_4}\nabla_v^{k_5}f\in C_p^{\a_1-\lfloor\a_1\rfloor,\a_2-\lfloor\a_2\rfloor,\a_3-\lfloor\a_3\rfloor,\a_4-\lfloor\a_4\rfloor,
\a_5-\lfloor\a_5\rfloor},
\]
where \(\lfloor\a_1\rfloor\) is the integral part of \(\a_1\), \(k_i\leq\lfloor\a_i\rfloor\) for \(1\leq i\leq5\) and \(k_1+k_2+k_3+k_4+k_5\leq\max_{1\leq i\leq5}{\lfloor\a_i\rfloor}\).
Replacing the subscript \(p\) by \(b\) indicates that the polynomial
growth condition is replaced by uniform boundedness. In particular,
all the indicated derivatives and H\"{o}lder regularity are uniformly
bounded in all variables.

\smallskip
The rest of the paper is organized as follows. In Section~\ref{sec:result}, we state the assumptions, construct the effective covariances and present the main theorem. Section~\ref{sec:est} establishes the exact decomposition and the preliminary estimates. Section~\ref{sec:slowcon} develops the fluctuation estimates and the joint limit for the slow pair. Section~\ref{sec:fast} proves the terminal local-equilibrium estimate. Section~\ref{sec:proof-main} combines these results to establish the quantitative fixed-time CLT.

\section{Assumptions and main result}\label{sec:result}

Throughout the paper, we impose the following assumptions.

\medskip
\noindent
{\bf(A1)}
There exist constants \(\lambda_\gamma,\lambda_\sigma>0\) such that
for every \(x,\xi\in\R^d\),
\begin{equation*}
\ip{\frac{\gamma(x)+\gamma^*(x)}{2}\xi}{\xi}
\geq\lambda_\gamma\norm{\xi}^2,
\qquad
\ip{\sigma(x)\sigma^*(x)\xi}{\xi}
\geq\lambda_\sigma\norm{\xi}^2.
\end{equation*}

\smallskip
\noindent
{\bf(A2)} For \(\beta\in(0,1]\), the coefficients satisfy
$
F\in C_b^{1+\beta}$ and $
\sigma, \gamma\in C_b^{2+\beta}$.

\begin{remark}
(i) The regularity in \textbf{(A2)} is dictated by the derivatives appearing in the effective drift of the limit equations. In particular, the noise-induced drift $\mathcal S$ contains $\nabla(\gamma^{-1})$, and its linearization contains $\nabla\mathcal S$, hence second derivatives of $\gamma$ are needed. Under {\bf(A1)}--{\bf(A2)} the effective coefficients satisfy $B,a\in C_b^{1+\beta}$, so the derivatives appearing in (\ref{eq:limitY}) are bounded and $\beta$-H\"older continuous. This regularity determines the algebraic rate $\eps^{\beta/2}$ in the weak error estimate.

(ii) The boundedness assumptions are imposed for simplicity of presentation. More generally, coefficients with growth can be treated under appropriate local regularity and growth conditions, provided that the systems are well posed and the moment estimates stated  Lemma \ref{lem:monest} hold; see e.g., \cite{HVW,XZ} for related assumptions.
\end{remark}

\subsection{The two effective covariances}

We first construct the two covariance matrices that determine the limiting fluctuation.

Recall that for each frozen position $x$, the velocity equation \eqref{eq:frozenv} admits an invariant measure $\zeta^x=\mathcal N(0,M(x))$, where $M(x)$ solves \eqref{eq:lyapM}. Define
\begin{equation}\label{eq:def-N}
N(x):=\gamma^{-1}(x)M(x)\gamma^{-*}(x).
\end{equation}
Then $N(x)$ is the unique symmetric positive definite solution of the Lyapunov equation
\begin{equation*}
\gamma(x)N(x)+N(x)\gamma^*(x)=a(x)a^*(x).
\end{equation*}
For $x,y\in\R^d$,  let us define the Gaussian kernel
\begin{equation}\label{muxy}
\mu^{x,y}:=\mathcal N(y,N(x)).
\end{equation}
One can check that it is the unique invariant measure of the frozen Ornstein--Uhlenbeck equation
\begin{equation}\label{eq:forz}
\dif Z_t^{x,y}
=
-\gamma(x)(Z_t^{x,y}-y)\dif t+a(x)\dif W_t.
\end{equation}
In the limiting description, the parameter $y$ represents the accumulated slow fluctuation, while \(N(x)\) is the covariance of the local Gaussian correction
that remains in the normalized position difference after the slow
fluctuation has been separated. This family $\mu^{x,y}$ will describe the conditional
fixed-time limit of the full position fluctuation.

The covariance of the accumulated slow fluctuation is obtained through a quadratic corrector.
Note that by \eqref{eq:def-S} and \eqref{eq:lyapM}, the function $h$, defined in \eqref{eq:defh}, is centered with respect to the frozen velocity equilibrium $\zeta^x$:
\begin{equation*}
\int_{\R^d}h(x,v)\,\zeta^x(\dif v)=0.
\end{equation*}
For $i=1,\dots,d$, define the symmetric matrix $\sS_i(x)$ by
\begin{equation*}
\sS_i(x):=\frac12\left(\nabla\left(\gamma_i^{-1}\right)(x)
+\left(\nabla\left(\gamma_i^{-1}\right)(x)\right)^*
\right),
\end{equation*}
where \(\gamma_i^{-1}\) denotes the \(i\)-th row of
\(\gamma^{-1}\). Let \(P_i(x)\) be the unique symmetric solution of
\begin{equation}\label{eq:lyapP}
\gamma^*(x)P_i(x)+P_i(x)\gamma(x)=\sS_i(x),
\end{equation}
and set
\[
\Phi_i(x,v)
:=
-v^*P_i(x)v+\Tr\big(P_i(x)M(x)\big).
\]
Then the vector $\Phi=(\Phi_1,\dots,\Phi_d)^*$ is centered under $\zeta^x$ and solves the quadratic Poisson equation
\begin{equation}\label{poissh}
\sL_0(x)\Phi(x,v)=-h(x,v).
\end{equation}
Indeed, by (\ref{eq:lyapM}) and (\ref{eq:lyapP}) we have $\Tr\big(P_i(\sigma\sigma^*)\big)=\Tr(\sS_iM)=\cS_i$, and hence
\begin{align*}
\sL_0(x)\Phi_i(x,v)&=2v^*P_i(x)\g(x)v-\Tr(P_i(x)\sigma(x)\sigma^*(x))\\
&=v^*\big(P_i\g+\g^*P_i\big)(x)v-\Tr(P_i\sigma\sigma^*)(x)\\
&=v^*\sS_i(x)v-\cS_i(x)=-h_i(x,v).
\end{align*}

At the fluctuation scale, the common-noise linearization in (\ref{eq:defY}) contributes the fast martingale coefficient $\nabla a(x)[\Gamma(x,v)]$, while the centered integral of $h$ contributes $\nabla_v\Phi(x,v)\sigma(x)$ through the quadratic Poisson corrector. Both terms are driven by the same Brownian motion and must therefore be combined before their quadratic variation is averaged. To express this sum explicitly, define the tensor
\(\cR(x)=(\cR_{ijk}(x))_{1\leq i,j,k\leq d}\) by
\begin{equation}\label{defr1}
\cR_{ijk}(x):=\sum_{\ell=1}^d\partial_{x_\ell}a_{ij}(x)(\gamma^{-1})_{\ell k}(x)
-2\sum_{m=1}^d(P_i(x))_{mk}\sigma_{mj}(x),
\end{equation}
and write
\begin{equation*}
(\cR(x)v)_{ij}:=\sum_{k=1}^d\cR_{ijk}(x)v_k.
\end{equation*}
Equivalently,
\[
\cR(x)v=\nabla a(x)[\gamma^{-1}(x)v]+\nabla_v\Phi(x,v)\sigma(x).
\]
Then the effective covariance $\cQ(x)$ of the additional slow noise is given by
\begin{align}\label{eq:defQ}
\cQ_{ir}(x)
&:=
\int_{\R^d}
\bigl((\cR(x)v)(\cR(x)v)^*\bigr)_{ir}\,\zeta^x(\dif v)
\no\\
&=
\sum_{j,k,\ell=1}^d
\cR_{ijk}(x)M_{k\ell}(x)\cR_{rj\ell}(x).
\end{align}
In particular, $\cQ(x)$ is symmetric and nonnegative definite, but need not be positive definite. The distinction is meaningful: the local velocity correction remains nondegenerate under \textbf{(A1)}, even when the accumulated fluctuation receives no additional independent noise.

\subsection{Main theorem}

Recall that $Z_t^\eps$ is defined by \eqref{eq:defZ}, the overdamped limit $\bar X_t$ satisfies \eqref{eq:overdamped} with $\bar X_0=x$, and $\bar Y_t$ solves \eqref{eq:limitY} with $\bar Y_0=0$. Although $\cQ$ may be degenerate, the equation for $\bar Y_t$ is linear once $\bar X_t$ is given, thus it is well defined.

We now state the main result of this paper.
\begin{theorem}\label{thm:main}
Assume {\bf(A1)} and {\bf(A2)} hold.  Then for every \(T>0\) and
\(\varphi\in C_b^{\a}\) with $\a\in(0,1)$, there exist
constants $C_{T,\varphi},C_\varphi,\kappa>0$, independent of
$\varepsilon\in(0,1)$ and $t\in(0,T]$,  such that
\begin{align}\label{eq:mainest}
\left|\E\varphi(Z_t^\eps)-\E\left(\int_{\R^d}\varphi(z)
\,\mu^{\bar X_t,\bar Y_t}(\dif z)\right)\right|\leq C_{T,\varphi}\, \eps^{\beta/2}+C_\varphi\,\mathrm e^{-\kappa t/\eps},
\end{align}
where $\mu^{x,y}$ is the Gaussian kernel given by (\ref{muxy}). Consequently, for every fixed \(t>0\),
\begin{equation}\label{eq:jointcon}
Z_t^\eps\Longrightarrow Z_t:=\bar Y_t+N^{1/2}(\bar X_t)\xi,
\end{equation}
where \(\xi\sim\mathcal N(0,\mI_d)\) is independent of \((\bar X_t,\bar Y_t)\).
\end{theorem}

\br[Conditional structure and failure of path-space convergence]
The representation \eqref{eq:jointcon} gives
\[
\Law(Z_t\mid\bar X_t,\bar Y_t)
=
\mathcal N(\bar Y_t,N(\bar X_t)).
\]
Thus the slow coordinate determines the random center of the limiting distribution, and the frozen velocity equilibrium determines its local covariance. This provides a basis for fluctuation corrections to reduced models.

The variable $\bar Y_t$ is the endpoint of a macroscopic diffusion and contains the accumulated second-order error. The Gaussian term $N^{1/2}(\bar X_t)\xi$ is an instantaneous local-equilibrium correction. In fact, the former admits a path-space limit, but the latter prevents such a limit for the full process $Z^\eps$. Hence $(Z^\eps)_{\eps>0}$ is generally not tight in $C([0,T];\R^d)$. The failure of a continuous-path CLT for $Z^\eps$ does not prevent a precise limiting description.
\er

\begin{remark}[Convergence rate and the initial layer]
Note that the convergence rate is independent of the regularity of the test function $\varphi$. When $\beta=1$, the algebraic term in \eqref{eq:mainest} is of order $\sqrt{\varepsilon}$, which should be optimal in terms of the first-order Taylor remainders and the central-limit scaling. For general $\beta\in(0,1]$, it quantifies how the coefficient regularity affects the accuracy of the fluctuation approximation.

An initial-layer contribution in the estimate is necessary and cannot in general be removed. Indeed, $Z_0^\eps=Y_0^\eps=0$, whereas the conditional equilibrium predicted by the limiting kernel at the initial slow state is $\mathcal N(0,N(x))$. The exponential term describes relaxation on the time scale $\varepsilon$. To make this term no larger than the algebraic error, it suffices to take
\[
t\geq\frac{\beta}{2\kappa}\,\varepsilon\log(1/\varepsilon).
\]
Thus the relaxation scale and the time needed to attain the stated algebraic accuracy are distinct.
\end{remark}

\begin{remark}[Cancellation in one dimension]
When $d=1$, the Lyapunov equations give
\[
M(x)=\frac{\sigma^2(x)}{2\gamma(x)},\qquad N(x)=\frac{\sigma^2(x)}{2\gamma^3(x)},
\qquad \cS(x)=-\frac{\gamma'(x)\sigma(x)^2}{2\gamma^3(x)},
\qquad P(x)=-\frac{\gamma'(x)}{2\gamma^3(x)}.
\]
Since
\[
\frac{a'(x)}{\g(x)}=\frac{\sigma'(x)}{\gamma^2(x)}-\frac{\sigma(x)\gamma'(x)}{\gamma^3(x)},
\qquad -2P(x)\sigma(x)=\frac{\sigma(x)\gamma'(x)}{\gamma^3(x)},
\]
the friction-derivative contributions cancel, so that
\[
\cR(x)=\frac{\sigma'(x)}{\gamma^2(x)}.
\]
Therefore we have
\begin{equation*}
\cQ(x)=\cR^2(x)M(x)=\frac{\sigma^2(x)(\sigma'(x))^2}{2\gamma^5(x)}.
\end{equation*}
The limit equation becomes
\[
\dif\bar Y_t=B'(\bar X_t)\bar Y_t\dif t+a'(\bar X_t)\bar Y_t\dif W_t+\frac{\sigma(\bar X_t)\sigma'(\bar X_t)}{\sqrt{2}\,\gamma^{5/2}(\bar X_t)}\dif\tilde W_t,
\qquad \bar Y_0=0.
\]
Thus derivatives of the friction coefficient cancel from the additional independent martingale coefficient, although they remain present in the common-noise linearization and in the linearized drift. In particular, if $\sigma$ is constant, then $\cQ\equiv0$ and $\bar Y_t\equiv0$, since $\bar Y_0=0$. The fixed-time limit nevertheless remains nontrivial:
\[
Z_t^\varepsilon
\Longrightarrow
N^{1/2}(\bar X_t)\xi.
\]
For constant friction as well, this reduces to $\mathcal N(0,\sigma^2/(2\gamma^3))$. Thus persistent terminal fluctuations can occur even when there is no accumulated slow fluctuation at the scale $\sqrt{\varepsilon}$.
\end{remark}

\section{Exact fluctuation decomposition and preliminary estimates}
\label{sec:est}

We first separate the persistent velocity contribution from the normalized position difference \(Z_t^\eps\).  The resulting  decomposition identifies the slow fluctuation variable that will be used throughout the proof.
Recall that
\[
\Gamma(x,v)=\gamma^{-1}(x)v.
\]
Since \(\nabla_v\Gamma=\gamma^{-1}\) and \(\nabla_v^2\Gamma=0\), by the definition \eqref{eq:L0} of the operator \(\sL_0(x)\), it holds that
\begin{equation}\label{eq:poivs}
\sL_0(x)\Gamma(x,v)=-v,
\end{equation}
As a result, we have the following decomposition.
\begin{lemma}\label{lem:decz}
Assume that {\bf(A1)} and {\bf(A2)} hold. Let \(Z_t^\eps\) and
\(Y_t^\eps\) be defined by \eqref{eq:defZ} and \eqref{eq:defY},
respectively. Then, for every \(t\geq0\),
\begin{equation}\label{eq:dec}
Z_t^\eps
=Y_t^\eps+\Gamma(X_t^\eps,V_t^\eps)-\sqrt{\eps}\,\Gamma(x,v).
\end{equation}
\end{lemma}

\begin{proof}
Applying It\^o's formula to
\(\Gamma(X_t^\eps,V_t^\eps)\) and by \eqref{eq:xv}, we have
\begin{align*}
\Gamma(X^\eps_t,V_t^\eps)&=\Gamma(x,\sqrt{\eps}v)+\frac{1}{\sqrt{\eps}}\int_0^t
\langle \nabla_x\Gamma(X^\eps_s,V_s^\eps),V_s^\eps\rangle\dif s
+\frac{1}{\eps}\int_0^t\sL_0(X^\eps_s)\Gamma(X^\eps_s,V_s^\eps)\dif s\\
&\quad+\frac{1}{\sqrt{\eps}}\int_0^t\langle \nabla_v\Gamma(X^\eps_s,V_s^\eps),F(X_s^\eps)\rangle\dif s+\frac{1}{\sqrt{\eps}}\int_0^t\langle \nabla_v\Gamma(X^\eps_s,V_s^\eps),\sigma(X_s^\eps)\dif W_s\rangle.
\end{align*}
Combining this identity with \eqref{eq:poivs}, we obtain
\begin{align}\label{eq:veps}
-\frac{1}{\eps}\int_0^tV_s^\eps\dif s
&=\Gamma(X^\eps_t,V_t^\eps)-\sqrt{\eps}\Gamma(x,v)-\frac{1}{\sqrt{\eps}}\int_0^t\nabla\big(\g^{-1}\big)(X^\eps_s)\cdot V_s^\eps(V_s^\eps)^*\dif s\no\\
&\quad-\frac{1}{\sqrt{\eps}}\int_0^t\g^{-1}(X^\eps_s)F(X_s^\eps)\dif s-\frac{1}{\sqrt{\eps}}\int_0^ta(X_s^\eps)\dif W_s,
\end{align}
where
\[
\left(\nabla\big(\gamma^{-1}\big)(x)\cdot vv^*\right)_i
:=\sum_{j,k=1}^d\partial_{x_j}(\gamma^{-1})_{ik}(x)v_kv_j.
\]
On the other hand, subtracting the position equation in \eqref{eq:xv} from \eqref{eq:overdamped}, we get that
\begin{equation}\label{eq:zeps}
Z_t^\eps=\frac1{\sqrt{\eps}}\int_0^tB(\bar X_s)\dif s-\frac1\eps\int_0^tV_s^\eps\,\dif s
+\frac1{\sqrt{\eps}}\int_0^ta(\bar X_s)\dif W_s.
\end{equation}
Substituting \eqref{eq:veps} into \eqref{eq:zeps}, and then using
\(B=\gamma^{-1}F+\cS\) and the definition \eqref{eq:defh} of \(h\), we obtain the desired decomposition of $Z_t^\eps$.
The proof is finished.
\end{proof}

We next collect the moment estimates used throughout the paper.

\begin{lemma}\label{lem:monest}
Assume {\bf(A1)} and {\bf(A2)} hold. Then for every \(q\geq2\) and \(T>0\),
there exists a constant \(C_{q,T}>0\) such that
\begin{align}
&\sup_{\eps\in(0,1)}\sup_{0\leq t\leq T}\E\norm{V_t^\eps}^{q}
\leq C_{q,T}\big(1+\norm v^q\big),\label{eq:estv}\\
&\sup_{\eps\in(0,1)}\sup_{0\leq t\leq T}\E\norm{X_t^\eps}^{q}
\leq C_{q,T}\big(1+\norm x^q+\norm v^q\big),\label{eq:estx}\\
&\sup_{\eps\in(0,1)}\sup_{0\leq t\leq T}\E\norm{Z_t^\eps}^{q}\leq C_{q,T},\label{eq:estz}
\end{align}
and
\begin{equation}\label{eq:esty}
\sup_{\eps\in(0,1)}\sup_{0\leq t\leq T}\E\norm{Y_t^\eps}^{q}\leq C_{q,T}.
\end{equation}
\end{lemma}

\begin{proof}
Applying It\^o's formula to \(\norm{V_t^\eps}^q\), and  using the
coercivity assumption of \(\gamma\), Young's inequality,
and the boundedness of \(F\) and \(\sigma\), we have
\[
\frac{\dif}{\dif t}\E\norm{V_t^\eps}^q
\leq-\frac{c_q}{\eps}\E\norm{V_t^\eps}^q+\frac{C_q}{\eps}.
\]
As a result of  the comparison principle, we obtain \eqref{eq:estv}.
Combining the position equation in \eqref{eq:xv} with \eqref{eq:veps}, we get
\begin{align*}
X_t^\eps&=x+\eps\g^{-1}(x)v-\sqrt{\eps}\g^{-1}(X^\eps_t)V_t^\eps+\int_0^t\nabla\big(\g^{-1}\big)(X^\eps_s)\cdot V_s^\eps(V_s^\eps)^*\dif s\\
&\quad+\int_0^t\g^{-1}(X^\eps_s)F(X_s^\eps)\dif s+\int_0^ta(X_s^\eps)\dif W_s.
\end{align*}
Estimate \eqref{eq:estx} then follows from \eqref{eq:estv}, Burkholder--Davis--Gundy's inequality, and the boundedness of the coefficients. Meanwhile, estimate \eqref{eq:estz} follows directly from the strong convergence estimate \eqref{eq:skrate}.
Finally,  by \eqref{eq:dec} and the boundedness
of \(\gamma^{-1}\), we deduce that
\[
\norm{Y_t^\eps}\leq\norm{Z_t^\eps}+C\norm{V_t^\eps}+C\sqrt{\eps}\norm v,
\]
which together with \eqref{eq:estv} and \eqref{eq:estz} proves \eqref{eq:esty}. The proof is finished.
\end{proof}

The following lemma provides a quantitative Taylor
remainder estimate  for \(f(\bar X_t)-f(X_t^\eps)\).

\begin{lemma}
Let \(f\in C_b^{1+\beta}\) with \(\beta\in(0,1]\). Then there exists a constant $C_0>0$ such that
\begin{align}\label{eq:rem2}
&\left|\frac{f(\bar X_t)-f(X_t^\eps)}{\sqrt{\eps}}-\nabla f(X_t^\eps)Y_t^\eps
-\nabla f(X_t^\eps)\Gamma(X_t^\eps,V_t^\eps)\right|
\leq C_0\,\eps^{\beta/2}\bigl(1+\norm{Z_t^\eps}^{1+\beta}\bigr).
\end{align}
\end{lemma}

\begin{proof}
Since $\bar X_t-X_t^\eps=\sqrt{\eps}\,Z_t^\eps$, by the mean value theorem, we have for some $0<\theta<1$,
\begin{align}\label{eq:rem1}
\left|\frac{f(\bar X_t)-f(X_t^\eps)}{\sqrt{\eps}}-\nabla f(X_t^\eps)Z_t^\eps\right|
&=\left|\nabla f(X_t^\eps+\theta\sqrt{\eps}Z_t^\eps)Z_t^\eps-\nabla f(X_t^\eps)Z_t^\eps\right|\no\\
&\leq C_1\, \eps^{\beta/2}\norm{Z_t^\eps}^{1+\beta}.
\end{align}
By the decomposition of $Z_t^\eps$ in Lemma~\ref{lem:decz}, we further deduce that
\begin{align*}
&\left|\nabla f(X_t^\eps)Z_t^\eps-\nabla f(X_t^\eps)Y_t^\eps-\nabla f(X_t^\eps)\Gamma(X_t^\eps,V_t^\eps)\right|\\
&=\sqrt{\eps}\left|\nabla f(X_t^\eps)\Gamma(x,v)\right|\leq C_2\sqrt{\eps}\norm v.
\end{align*}
Combining this  with \eqref{eq:rem1}, we obtain \eqref{eq:rem2}. The proof is finished.
\end{proof}

The following lemma establishes several regularity results concerning the coefficients, which will be used repeatedly below.

\begin{lemma}\label{lem:reg}
Assume that {\bf(A1)} and {\bf(A2)} hold. Then we have
\[
\cS,B,P_i,\cR,\cQ\in C_b^{1+\beta},
\qquad
a,M,N\in C_b^{2+\beta}.
\]
Moreover, \(M\) and \(N\) are uniformly positive definite.
\end{lemma}

\begin{proof}
The solution of the Lyapunov equation \eqref{eq:lyapM} admits the
representation
\[
M(x)=\int_0^{+\infty}\mathrm e^{-s\gamma(x)}\sigma(x)\sigma^*(x)\mathrm e^{-s\gamma^*(x)}\dif s.
\]
Applying Duhamel's formula and using the uniform coercivity in assumption {\bf(A1)}, we have for every \(y \in \mathbb{R}^d\),
\begin{align*}
\left| \nabla_x \mathrm e^{-s\gamma(x)}[y] \right| &= \left| - \int_0^s\mathrm e^{-(s-r)\gamma(x)} \nabla\gamma(x)[y]\mathrm e^{-r\gamma(x)} \dif r \right| \\
&\le C_0 |y| \int_0^s\mathrm e^{-\lambda_\gamma(s-r)}\cdot\mathrm e^{-\lambda_\gamma r} \dif r \\
&= C_0 s \mathrm e^{-\lambda_\gamma s} |y|,
\end{align*}
which together with \(\gamma,\sigma \in C_b^{2+\beta}\)  yields \(M \in C_b^{2+\beta}\). Since \(\gamma \) is uniformly invertible, its inverse satisfies \(\gamma^{-1} \in C_b^{2+\beta}\).
It then follows from \eqref{eq:defBa}, \eqref{eq:def-S} and
\eqref{eq:def-N} that
\[
\cS,B\in C_b^{1+\beta},\qquad a,N\in C_b^{2+\beta}.
\]
Similarly, by (\ref{eq:lyapP}), \(P_i\) admits the representation
\[
P_i(x)=\int_0^{+\infty}\mathrm e^{-s\gamma^*(x)}\sS_i(x)\mathrm e^{-s\gamma(x)}\dif s.
\]
Since \(\sS_i\in C_b^{1+\beta}\), using the same argument as above, we have
\(P_i\in C_b^{1+\beta}\).
Finally, formula \eqref{defr1} implies that
\(\cR\in C_b^{1+\beta}\), and \eqref{eq:defQ} then gives
\(\cQ\in C_b^{1+\beta}\). Moreover, the uniform positive definiteness of \(M\) follows from the uniform
ellipticity of \(\sigma\sigma^*\), while that of \(N\) follows from
\eqref{eq:def-N}. The proof is finished.
\end{proof}

\section{Joint diffusion limit for the slow fluctuation}
\label{sec:slowcon}
In this section, we first establish quantitative fluctuation estimates for centered functions.
Then we use the Kolmogorov equation to prove the weak convergence of the slow pair.

\subsection{Second-order fluctuation estimates}

We shall need some mollification arguments due to the low regularity of the coefficients.
Let \(\rho_1:\R\to[0,1]\) and \(\rho_2:\R^d\to[0,1]\) be two smooth radial convolution kernel functions such that \(\int_{\R}\rho_1(r)\dif r=\int_{\R^d}\rho_2(x)\dif x=1\), and for any \(k\geq1\), there exist constants \(C_k>0\) such that \(|\nabla^k\rho_1(r)|\leq C_k\rho_1(r)\) and \(|\nabla^k\rho_2(x)|\leq C_k\rho_2(x)\). For every \(n\geq1\), define
\[
\rho^n_1(r):=n^2\rho_1(n^2r)\quad\mathrm{and}\quad\rho^n_2(x):=n^{d}\rho_2(nx).
\]
For a function \(f(t,x,y,v)\), define
\begin{align}\label{eq:fn}
f_n(t,x,y,v)
:=
\int_{\R^{2d+1}}
&f(t-s,x-\tilde x,y-\tilde y,v)
\rho_1^n(s)\rho_2^n(\tilde x)\rho_2^n(\tilde y)
\,\dif s\,\dif\tilde x\,\dif\tilde y.
\end{align}
The proof of the following  mollification estimates are standard, see e.g.,  \cite[Lemma 4.1]{RX}.

\begin{lemma}\label{lem:molf}
Let \(f\in C_p^{\vartheta/2,\vartheta,\vartheta,0}\) with \(\vartheta\in(0,2]\), and let \(f_n\) be defined by \eqref{eq:fn}. Then for every \(y,v\in\R^d\),
\begin{align}
&\|f_n(\cdot,\cdot,y,v)-f(\cdot,\cdot,y,v)\|_\infty
\leq C_0n^{-\vartheta}\left(1+\norm y^p+\norm v^p\right),\label{eq:mol0}\\
&\|\nabla_xf_n(\cdot,\cdot,y,v)\|_\infty+\|\nabla_yf_n(\cdot,\cdot,y,v)\|_\infty
\leq C_0n^{1-(\vartheta\wedge1)}\left(1+\norm y^p+\norm v^p\right),\label{eq:mol1}
\end{align}
and
\begin{align}\label{eq:mol2}
&\|\p_tf_n(\cdot,\cdot,y,v)\|_\infty+\|\nabla^2_xf_n(\cdot,\cdot,y,v)\|_\infty
+\|\nabla^2_yf_n(\cdot,\cdot,y,v)\|_\infty\no\\
&\leq C_0n^{2-\vartheta}\left(1+\norm y^p+\norm v^p\right),
\end{align}
where \(C_0>0\) is a constant independent of \(n\).
\end{lemma}

\vspace{2mm}
Let \(f(t,x,y,v)\) satisfy the centering condition
\begin{equation}\label{eq:cenf}
\int_{\R^d}f(t,x,y,v)\,\zeta^x(\dif v)=0, \qquad \forall(t,x,y)\in\R_+\times\R^{2d},
\end{equation}
where $\zeta^x=\cN(0,M(x))$ is the unique invariant measure of the frozen process \eqref{eq:frozenv}. Consider the following Poisson equation:
\begin{equation}\label{eq:poif}
\sL_0(x)\Psi(t,x,y,v)=-f(t,x,y,v),
\end{equation}
where the operator $\sL_0(x)$ is defined by \eqref{eq:L0} and $(t,x,y)$ are regarded as parameters. For notational convenience, set
\begin{align*}
\cH_f(t,x,y,v)&:=
\big\langle \nabla_x\Psi(t,x,y,v),v\big\rangle+\big\langle \nabla_v\Psi(t,x,y,v),F(x)\big\rangle\no\\
&\quad+\big\langle \nabla_y\Psi(t,x,y,v),h(x,v)\big\rangle\no\\
&\quad+\nabla_{yv}^2\Psi(t,x,y,v):\big(\nabla a(x)\big[y+\gamma^{-1}(x)v\big]\sigma^*(x)\big),
\end{align*}
and define its average by
\begin{align*}
\bar\cH_f(t,x,y):=\int_{\R^d}\cH_f(t,x,y,v)\,\zeta^x(\dif v),
\end{align*}
where \(h(x,v)\) and \(\nabla a(x)[y+\gamma^{-1}(x)v]\) are defined by \eqref{eq:defh} and \eqref{eq:day}, respectively.

\vspace{2mm}

The following proposition provides the foundational fluctuation estimates for centered functions.

\begin{proposition}\label{prop:flu}
Assume {\bf(A1)} and {\bf(A2)} hold, and suppose that \(f\) satisfies
\eqref{eq:cenf}.

\vspace{1mm}
\noindent(i) If \(f\in C_p^{\b/2,\b,\b,\d}\) with \(\b,\d\in (0,1]\), then we have
\begin{equation}\label{eq:est1}
\left|\E\left(\int_{0}^{t}f(s,X^{\eps}_{s},Y_s^\eps,V^{\eps}_{s})\dif s\right)\right|
\leq C_t\,\eps^{\b/2}.
\end{equation}

\vspace{1mm}
\noindent(ii) If \(f\in C_p^{(1+\b)/2,1+\b,1+\b,\d}\) with \(\b,\d\in (0,1]\), then we have
\begin{equation}\label{eq:est2}
\left|\E\left(\frac{1}{\sqrt{\eps}}\int_{0}^{t}f(s,X^{\eps}_{s},Y_s^\eps,V^{\eps}_{s})\dif s\right)
-\E\left(\int_{0}^{t}\bar{\cH}_f(s,X^{\eps}_{s},Y^{\eps}_{s})\dif s\right)\right|
\leq C_t\,\eps^{\b/2},
\end{equation}
where $C_t>0$ is a constant independent of \(\eps\) and \(\d\).
\end{proposition}

\begin{proof}  We divide the proof into two steps.

\medskip
\noindent
\emph{Step 1: Proof of \eqref{eq:est1}.}
Since \(f\) satisfies \eqref{eq:cenf}, by \cite[Theorem 3.1]{XY}, there exists a unique centered solution $\Psi\in C_{p}^{\b/2,\b,\b,2+\d}$ to the Poisson equation \eqref{eq:poif}. Let \(\Psi_n\) be the mollifying approximation of \(\Psi\) defined as in \eqref{eq:fn}.
By It\^o's formula and using the Poisson equation \eqref{eq:poif}, we deduce that
\begin{align}\label{eq:est3}
&\E\left(\int_{0}^{t}f(s,X^{\eps}_{s},Y_s^\eps,V^{\eps}_{s})\dif s\right)\no\\
&=\eps\E\left[\Psi_n(0,x,0,\sqrt{\eps}v)
-\Psi_n(t,X^{\eps}_{t},Y^{\eps}_{t},V^{\eps}_{t})+\int_{0}^{t}\p_{s}\Psi_n(s,X^{\eps}_{s},Y^{\eps}_{s},V^{\eps}_{s})\dif s\right]\no\\
&\quad+\sqrt{\eps}\E\left(\int_{0}^{t}\big[\big\langle \nabla_x\Psi_n(s,X^{\eps}_{s},Y^{\eps}_{s},V^{\eps}_{s}),
V^{\eps}_{s}\big\rangle+\big\langle \nabla_v\Psi_n(s,X^{\eps}_{s},Y^{\eps}_{s},V^{\eps}_{s}), F(X^{\eps}_{s})\big\rangle\big]\dif s\right)\no\\
&\quad+\sqrt{\eps}\E\left(\int_{0}^{t}\big\langle \nabla_y\Psi_n(s,X^{\eps}_{s},Y^{\eps}_{s},V^{\eps}_{s}),
B(\bar X_s)-B(X_s^\eps)\big\rangle\dif s\right)\no\\
&\quad+\sqrt{\eps}\E\left(\int_{0}^{t}\big\langle \nabla_y\Psi_n(s,X^{\eps}_{s},Y^{\eps}_{s},V^{\eps}_{s}),
h(X_s^\eps,V_s^\eps)\big\rangle\dif s\right)\no\\
&\quad+\frac{1}{2}\E\left(\int_{0}^{t}\nabla^2_{y}\Psi_n(s,X^{\eps}_{s},Y^{\eps}_{s},V^{\eps}_{s}):\big[\big(a(\bar X_s)-a(X^\eps_s)\big)\big(a(\bar X_s)-a(X^\eps_s)\big)^*\big]\dif s\right)\no\\
&\quad+\E\left(\int_{0}^{t}\nabla^2_{yv}\Psi_n(s,X^{\eps}_{s},Y^{\eps}_{s},V^{\eps}_{s}):\big[\big(a(\bar X_s)-a(X^\eps_s)\big)\sigma^*(X^\eps_s)\big]\dif s\right)\no\\
&\quad+\E\left(\int_{0}^{t}\big(\sL_{0}
\Psi_n-\sL_{0}\Psi\big)(s,X^{\eps}_{s},Y^{\eps}_{s},V^{\eps}_{s})\dif s\right).
\end{align}
Since \(\Psi\) has polynomial growth in \(y\) and \(v\), by estimates \eqref{eq:estv}, \eqref{eq:esty} and \eqref{eq:mol2},
the first term on the right-hand side of \eqref{eq:est3} is bounded
by
\[
C_t\,\eps\big(1+n^{2-\b}\big).
\]
The next three terms, each of which contains an explicit factor
\(\sqrt{\eps}\), are bounded using \eqref{eq:mol1} and Lemma~\ref{lem:reg} by
\[
C_t\sqrt\eps\big(1+n^{1-\b}\big).
\]
Moreover, by the mean value theorem and the boundedness of \(\nabla a\), we have
\begin{equation}\label{eq:estz1}
\big|a(\bar X_s)-a(X_s^\eps)\big|\leq C_1\big|\bar X_s-X_s^\eps\big|=C_1\sqrt{\eps}\big|Z_s^\eps\big|.
\end{equation}
As a result, the two terms involving \(\nabla_y^2\Psi_n\) and
\(\nabla_{yv}^2\Psi_n\) are bounded, respectively, by
\[
C_t\eps n^{2-\beta} \qquad\text{and}\qquad C_t\sqrt{\eps}\,n^{1-\beta}.
\]
Finally,  since \(\nabla_v\Psi, \nabla_{v}^2\Psi\in C_{p}^{\b/2,\b,\b,\d}\) and mollification is performed only in \((t,x,y)\), we have
\[
\nabla_v\Psi_n=(\nabla_v\Psi)_n\qquad\text{and}\qquad \nabla_{v}^2\Psi_n=(\nabla_{v}^2\Psi)_n.
\]
It follows from \eqref{eq:mol0} and the moment estimates that
\begin{align*}
&\left|\E\left(\int_{0}^{t}\big(\sL_{0}
\Psi_n-\sL_{0}\Psi\big)(s,X^{\eps}_{s},Y^{\eps}_{s},V^{\eps}_{s})\dif s\right)\right|\\
&\leq C_2\,n^{-\b}\int_{0}^{t}\big(1+\E\big|Y^{\eps}_{s}\big|^p+\E\big|V^{\eps}_{s}\big|^p\big)\dif s\leq C_t\,n^{-\b}.
\end{align*}
Combining the above estimates with \eqref{eq:est3}, we arrive at
\begin{equation*}
\left|\E\left(\int_{0}^{t}f(s,X^{\eps}_{s},Y_s^\eps,V^{\eps}_{s})\dif s\right)\right|
\leq C_t\big(\eps+\sqrt{\eps}+\eps n^{2-\b}+\sqrt\eps n^{1-\b}+n^{-\b}\big).
\end{equation*}
Taking \(n=\eps^{-1/2}\) and using \(\beta\in(0,1]\),  we get \eqref{eq:est1}.

\medskip
\noindent
\emph{Step 2: Proof of \eqref{eq:est2}.}
We now assume that \(f\in C_p^{(1+\beta)/2,1+\beta,1+\beta,\delta}\). Then the corresponding solution of \eqref{eq:poif} satisfies \(\Psi\in C_p^{(1+\beta)/2,1+\beta,1+\beta,2+\delta}\).
Dividing \eqref{eq:est3} by \(\sqrt{\eps}\) and retaining the leading-order terms, we have
\begin{align}\label{eq:est4}
&\E\left(\frac{1}{\sqrt{\eps}}\int_{0}^{t}f(s,X^{\eps}_{s},Y_s^\eps,V^{\eps}_{s})\dif s\right)\no\\
&=\E\left(\int_{0}^{t}\big[\big\langle \nabla_x\Psi_n(s,X^{\eps}_{s},Y^{\eps}_{s},V^{\eps}_{s}),
V^{\eps}_{s}\big\rangle+\big\langle \nabla_v\Psi_n(s,X^{\eps}_{s},Y^{\eps}_{s},V^{\eps}_{s}), F(X^{\eps}_{s})\big\rangle\big]\dif s\right)\no\\
&\quad+\E\left(\int_{0}^{t}\big\langle \nabla_y\Psi_n(s,X^{\eps}_{s},Y^{\eps}_{s},V^{\eps}_{s}),
h(X_s^\eps,V_s^\eps)\big\rangle\dif s\right)\no\\
&\quad+\E\left(\int_{0}^{t}\nabla^2_{yv}\Psi_n(s,X^{\eps}_{s},Y^{\eps}_{s},V^{\eps}_{s}):
\left[\frac{1}{\sqrt{\eps}}\big(a(\bar X_s)-a(X^\eps_s)\big)\sigma^*(X^\eps_s)\right]\dif s\right)+\sR_t^{\eps,n},
\end{align}
where the remaining term $\sR_t^{\eps,n}$ is defined by
\begin{align*}
\sR_t^{\eps,n}&:=\sqrt{\eps}\E\left[\Psi_n(0,x,0,\sqrt{\eps}v)
-\Psi_n(t,X^{\eps}_{t},Y^{\eps}_{t},V^{\eps}_{t})+\int_{0}^{t}\p_{s}\Psi_n(s,X^{\eps}_{s},Y^{\eps}_{s},V^{\eps}_{s})\dif s\right]\\
&\quad+\E\left(\int_{0}^{t}\big\langle \nabla_y\Psi_n(s,X^{\eps}_{s},Y^{\eps}_{s},V^{\eps}_{s}),
B(\bar X_s)-B(X_s^\eps)\big\rangle\dif s\right)\\
&\quad+\frac{1}{2\sqrt{\eps}}\E\left(\int_{0}^{t}\nabla^2_{y}\Psi_n(s,X^{\eps}_{s},Y^{\eps}_{s},V^{\eps}_{s}):\big[\big(a(\bar X_s)-a(X^\eps_s)\big)\big(a(\bar X_s)-a(X^\eps_s)\big)^*\big]\dif s\right)\\
&\quad+\frac{1}{\sqrt{\eps}}\E\left(\int_{0}^{t}\big(\sL_{0}
\Psi_n-\sL_{0}\Psi\big)(s,X^{\eps}_{s},Y^{\eps}_{s},V^{\eps}_{s})\dif s\right).
\end{align*}
Using Lemma~\ref{lem:molf}, \eqref{eq:estz1} and the same arguments
as in Step~1, we can get
\begin{equation}\label{eq:est5}
|\sR_t^{\eps,n}|\leq C_t\left(\sqrt{\eps}+\sqrt{\eps} n^{1-\b}+\frac{1}{\sqrt{\eps}}n^{-1-\b}\right).
\end{equation}
Applying \eqref{eq:rem2} to \(a\), we have
\begin{equation}\label{eq:alin}
\frac{1}{\sqrt{\eps}}\big(a(\bar X_s)-a(X^\eps_s)\big)
=\nabla a(X^\eps_s)[Y^\eps_s+\Gamma(X^\eps_s,V^\eps_s)]+r_s^\eps,
\end{equation}
where
\begin{equation}\label{eq:alinrem}
\big|r_s^\eps\big|\leq C\eps^{\beta/2}\big(1+\norm{Z_s^\eps}^{1+\beta}\big).
\end{equation}
Substituting \eqref{eq:alin} into \eqref{eq:est4}, and replacing the
derivatives of \(\Psi_n\) in the leading-order terms by those of
\(\Psi\), we have
\begin{equation}\label{eq:est6}
\E\left(\frac1{\sqrt\eps}\int_0^tf(s,X_s^\eps,Y_s^\eps,V_s^\eps)\dif s\right)
=\E\left(\int_0^t\cH_f(s,X_s^\eps,Y_s^\eps,V_s^\eps)\dif s\right)+\tilde\sR_t^{\eps,n}.
\end{equation}
The new remainder consists of \(\sR_t^{\eps,n}\), the Taylor
remainder in \eqref{eq:alin}, and the mollification errors in the
derivatives of \(\Psi\). More precisely,
\begin{align*}
\tilde\sR_t^{\eps,n}&:=\sR_t^{\eps,n}+\E\left(\int_{0}^{t}\nabla^2_{yv}\Psi_n(s,X^{\eps}_{s},Y^{\eps}_{s},V^{\eps}_{s}):
\big[r_s^\eps\sigma^*(X^\eps_s)\big]\dif s\right)\\
&\quad+\E\left(\int_{0}^{t}\big\langle (\nabla_x\Psi_n-\nabla_x\Psi)(s,X^{\eps}_{s},Y^{\eps}_{s},V^{\eps}_{s}),
V^{\eps}_{s}\big\rangle\dif s\right)\\
&\quad+\E\left(\int_{0}^{t}\big\langle(\nabla_v\Psi_n-\nabla_v\Psi)(s,X^{\eps}_{s},Y^{\eps}_{s},V^{\eps}_{s}), F(X^{\eps}_{s})\big\rangle\dif s\right)\\
&\quad+\E\left(\int_{0}^{t}\big\langle (\nabla_y\Psi_n-\nabla_y\Psi)(s,X^{\eps}_{s},Y^{\eps}_{s},V^{\eps}_{s}),
h(X_s^\eps,V_s^\eps)\big\rangle\dif s\right)\\
&\quad+\E\left(\int_{0}^{t}(\nabla^2_{yv}\Psi_n-\nabla^2_{yv}\Psi)(s,X^{\eps}_{s},Y^{\eps}_{s},V^{\eps}_{s}):
\big(\nabla a(X^\eps_s)[Y_s^\eps+\Gamma(X^\eps_s,V^\eps_s)]\sigma^*(X^\eps_s)\big)\dif s\right).
\end{align*}
Since \(\nabla_{yv}^2\Psi\in C_p^{\b/2,\b,\b,\d}\), the moment estimates and \eqref{eq:alinrem} imply
\begin{equation*}
\left|\E\left(\int_{0}^{t}\nabla^2_{yv}\Psi_n(s,X^{\eps}_{s},Y^{\eps}_{s},V^{\eps}_{s}):
\big[r_s^\eps\sigma^*(X^\eps_s)\big]\dif s\right)\right|\leq C_t\,\eps^{\beta/2}.
\end{equation*}
At the same time, applying \eqref{eq:mol0} to \(\nabla_x\Psi,\nabla_y\Psi,\nabla_v\Psi\) and \(\nabla_{yv}^2\Psi\), we have that the remaining mollification errors are  bounded
by \(C_tn^{-\beta}\). Together with \eqref{eq:est5}, this gives
\begin{equation}\label{eq:est7}
|\tilde{\sR}_t^{\eps,n}|\leq C_t\left(\sqrt{\eps}+\eps^{\beta/2}+\sqrt{\eps} n^{1-\b}+\frac{1}{\sqrt{\eps}}n^{-1-\b}+n^{-\b}\right).
\end{equation}
Finally, since \(\cH_f\in C_p^{\b/2,\b,\b,\d}\) and note that the function $\cH_f(t,x,y,v)-\bar{\cH}_f(t,x,y)$ satisfies the centering condition \eqref{eq:cenf},
we derive by \eqref{eq:est1} that
\begin{align*}
\left|\E\left(\int_0^t\big[\cH_f(s,X_s^\eps,Y_s^\eps,V_s^\eps)
-\bar{\cH}_f(s,X_s^\eps,Y_s^\eps)\big]\dif s\right)\right|
\leq C_t\,\eps^{\b/2}.
\end{align*}
Combining this estimate with \eqref{eq:est6} and \eqref{eq:est7}, we conclude that
\begin{align*}
&\left|\E\left(\frac{1}{\sqrt{\eps}}\int_{0}^{t}f(s,X^{\eps}_{s},Y_s^\eps,V^{\eps}_{s})\dif s\right)
-\E\left(\int_{0}^{t}\bar{\cH}_f(s,X^{\eps}_{s},Y^{\eps}_{s})\dif s\right)\right|\\
&\leq\left|\E\left(\int_0^t\big[\cH_f(s,X_s^\eps,Y_s^\eps,V_s^\eps)
-\bar{\cH}_f(s,X^{\eps}_{s},Y^{\eps}_{s})\big]\dif s\right)\right|+|\tilde{\sR}_t^{\eps,n}|\\
&\leq C_t\left(\sqrt{\eps}+\eps^{\beta/2}+\sqrt{\eps} n^{1-\b}+\frac{1}{\sqrt{\eps}}n^{-1-\b}+n^{-\b}\right).
\end{align*}
Taking \(n=\eps^{-1/2}\), we obtain \eqref{eq:est2}.
The proof is finished.
\end{proof}

\subsection{Joint diffusion limit of the slow process \((X_t^\eps,Y_t^\eps)\)}
Recall that $\bar X_t$ and $\bar Y_t$ satisfy (\ref{eq:overdamped}) and (\ref{eq:limitY}), respectively.
The  pair \((\bar X_t,\bar Y_t)\) is the solution of
\begin{equation}\label{eq:limit}
\left\{
\begin{aligned}
&\dif\bar X_t=B(\bar X_t)\,\dif t+a(\bar X_t)\dif W_t,\qquad\qquad\qquad\qquad\qquad\qquad\,\bar X_0=x,\\
&\dif\bar Y_t=\nabla B(\bar X_t)\bar Y_t\dif t+\nabla a(\bar X_t)[\bar Y_t]\dif W_t
+\sqrt{\cQ(\bar X_t)}\dif\tilde W_t,\qquad\bar Y_0=0,
\end{aligned}
\right.
\end{equation}
where \(\tilde W\) is a Brownian motion independent of \(W\).  The infinitesimal generator of \((\bar X_t,\bar Y_t)\) is given by
\begin{align}\label{eq:genxy}
\bar{\sL}
&:=\big\langle B(x),\nabla_x\big\rangle+\big\langle \nabla B(x)y,\nabla_y\big\rangle\no\\
&\quad+\frac12a(x)a^*(x):\nabla_{x}^2+a(x)\big(\nabla a(x)[y]\big)^*:\nabla_{xy}^2\no\\
&\quad+\frac12\left(\nabla a(x)[y]\big(\nabla a(x)[y]\big)^*+\cQ(x)\right):\nabla_{y}^2.
\end{align}
 Note that the mixed
second-order term in \eqref{eq:genxy} comes from the common Brownian
motion \(W\) driving both components of \eqref{eq:limit}.

To establish the weak convergence of \((X_t^\eps,Y_t^\eps)\), we use the Kolmogorov equation associated
with \(\bar\sL\). More precisely,
for \(T>0\) and \(\phi\in C_b^{2+\beta,2+\beta}\), consider the Cauchy problem
\begin{equation}\label{eq:cau}
\left\{ \begin{aligned}
&\p_t u(t,x,y)-\bar{\sL} u(t,x,y)=0,\quad t\in(0,T],\\
&u(0,x,y)=\phi(x,y).
\end{aligned} \right.
\end{equation}
Recall that $\cQ$ may be degenerate. We have the following  result.

\begin{lemma}\label{lem:regu}
Assume {\bf(A1)} and {\bf(A2)} hold. Then, for every \(T>0\) and every
\(\phi\in C_b^{2+\beta,2+\beta}\), the Cauchy problem \eqref{eq:cau}
admits a unique   solution \(u\in C_p^{1+\beta/2,\,2+\beta,\,2+\beta}\)
which is given by
\begin{equation}\label{eq:regu-representation}
u(t,x,y)=\E\phi\big(\bar X_t(x),\bar Y_t(x,y)\big),
\end{equation}
where \((\bar X_t(x),\bar Y_t(x,y))\) is the solution of \eqref{eq:limit} starting from \((\bar X_0,\bar Y_0)=(x,y)\).

Moreover, there exist \(p\geq1\) and \(C_T>0\) such that  for all \(t\in[0,T]\) and \(x,y\in\R^d\),
\begin{equation}\label{eq:regu-growth}
\begin{aligned}
&|\partial_tu(t,x,y)|
+|\nabla_x u(t,x,y)|+|\nabla_y u(t,x,y)|+|\nabla_x^2 u(t,x,y)|\\
&+|\nabla_y^2 u(t,x,y)|+|\nabla_{xy}^2 u(t,x,y)|
\leq C_T(1+|y|^p),
\end{aligned}
\end{equation}
In addition, for \(H\in\{\nabla_{x}^2u,\nabla_{xy}^2u,\nabla_{y}^2u\}\),
the following estimates hold:  for  \(x,x',y,y'\in\R^d\),
\begin{align}
|H(t,x,y)-H(t,x',y)|
&\leq C_T(1+|y|^p)|x-x'|^\beta,\label{eq:regu-holder-x}\\
|H(t,x,y)-H(t,x,y')|&\leq C_T(1+|y|^p+|y'|^p)|y-y'|^\beta.\label{eq:regu-holder-y}
\end{align}
The analogous parabolic H\"older estimates in the time variable also
hold, with exponent \(\beta/2\).
\end{lemma}

\begin{proof}
The uniqueness and the probabilistic representation \eqref{eq:regu-representation} follow from It\^o's formula. Below we prove the stated regularity estimates. We only present the main steps,  emphasizing why the possible degeneracy
in the \(y\)-variables does not prevent the stated regularity.

Let us first consider the
$y$-direction. For simplicity, write
\[
A(x):=\nabla B(x),
\qquad
G(x)[y]:=\nabla a(x)[y].
\]
Then the limiting equation for the second component can be written as
\[
\dif\bar Y_t
=
A(\bar X_t)\bar Y_t\,\dif t
+
G(\bar X_t)[\bar Y_t]\,\dif W_t
+
\sqrt{\cQ(\bar X_t)}\,\dif\tilde W_t.
\]
Since the equation for \(\bar X_t\) is independent of \(\bar Y_t\),  we have
\begin{equation}\label{eq:Y-affine}
\bar Y_t(x,y)
=
J_t^x y+\bar Y_t(x,0),
\end{equation}
where \(J_t^x\) is the matrix valued solution of
\[
\dif J_t^x
=
A(\bar X_t(x))J_t^x\,\dif t
+
\nabla a(\bar X_t(x))[J_t^x]\,\dif W_t,
\qquad
J_0^x=\mI_d.
\]
Since \(A\) and \(\nabla a\) are bounded,  by standard moment estimates we have for every \(q\geq2\),
\[
\sup_{0\leq t\leq T}\E|J_t^x|^q\leq C_{q,T},
\qquad
\sup_{0\leq t\leq T}
\E|\bar Y_t(x,y)|^q
\leq C_{q,T}(1+|y|^q).
\]
Differentiating the representation \eqref{eq:regu-representation}
with respect to \(y\)-variable  and using \eqref{eq:Y-affine}, we obtain
\begin{align*}
\nabla_yu(t,x,y)
&=
\E\left[
\nabla_y\phi\big(\bar X_t(x),\bar Y_t(x,y)\big)J_t^x
\right],\\
\nabla_{y}^2u(t,x,y)
&=
\E\left[
(J_t^x)^*
\nabla_{y}^2\phi\big(\bar X_t(x),\bar Y_t(x,y)\big)J_t^x
\right].
\end{align*}
These identities prove the required \(y\)-regularity and the
corresponding polynomial bounds. Notice that no regularity is gained
in the \(y\)-direction: the solution inherits its \(C^{2+\beta}\)
regularity in \(y\) directly from the terminal function \(\phi\).

We next consider the \(x\)-direction, and we focus on the a priori estimates. By {\bf(A1)}, \(a(x)a^*(x)\) is uniformly  elliptic. To handle the possible degeneracy in \(y\),
we differentiate the equation once and twice with respect to \(y\).
Since the coefficients of \(\bar{\mathscr L}\) are at most
linear in \(y\) in the drift and mixed-diffusion terms, and at most
quadratic in \(y\) in the \(\nabla_{y}^2\)-term, the resulting system for \(u\), \(\nabla_y u\), and \(\nabla_y^2 u\) is triangular, and its principal part in the \(x\)-variables remains
\[
\frac12 a(x)a^*(x):\nabla_x^2,
\]
which is still uniformly elliptic. Applying the parameter-dependent Schauder estimates for uniformly
elliptic systems  and  using the moment bounds, we have
\begin{equation*}
\begin{aligned}
\nabla_yu\in C_p^{(1+\beta)/2,\,1+\beta,\,1+\beta},\,\,
\nabla_{y}^2u\in C_p^{\beta/2,\,\beta,\,\beta}.
\end{aligned}
\end{equation*}
In particular, all terms containing \(y\)-derivatives in
\(\bar{\mathscr L}u\), including
\[
\langle A(x)y,\nabla_yu\rangle,
\qquad
a(x)\big(G(x)[y]\big)^*:\nabla_{xy}^2u,
\]
and
\[
\frac12
\left(
G(x)[y]\big(G(x)[y]\big)^*+\cQ(x)
\right):\nabla_{y}^2u,
\]
are \(\beta\)-H\"older continuous in \(x\), with bounds of the form
\(C_T(1+|y|^p)\).
Then for each fixed \(y\), we   rewrite \eqref{eq:cau}   as
\begin{equation}\label{eq:regu-x-equation}
\begin{aligned}
\partial_tu
-\langle B(x),\nabla_xu\rangle
-\frac12a(x)a^*(x):\nabla_{x}^2u
={}&
\langle A(x)y,\nabla_yu\rangle+
a(x)\big(G(x)[y]\big)^*:\nabla_{xy}^2u\\
&+
\frac12
\left(
G(x)[y]\big(G(x)[y]\big)^*+\cQ(x)
\right):\nabla_{y}^2u.
\end{aligned}
\end{equation}
The right-hand side belongs to the weighted parabolic
\(C^{\beta/2,\beta}\)-class in \((t,x)\). Since
\[
B,a\in C_b^{1+\beta},
\qquad
A=\nabla B,\ \nabla a,\ \cQ\in C_b^\beta,
\]
the classical Schauder estimate applied to
\eqref{eq:regu-x-equation} gives that
$
u\in C_p^{1+\beta/2,\,2+\beta,\,2+\beta},
$
as well as the estimates \eqref{eq:regu-growth}--\eqref{eq:regu-holder-y}.
 This completes the
proof.
\end{proof}

Now we prove  the weak convergence of
\((X_t^\varepsilon,Y_t^\varepsilon)\).

\begin{proposition}\label{prop:slow}
Assume {\bf(A1)} and {\bf(A2)} hold. Then, for every \(T>0\) and \(\phi\in C_b^{2+\b,2+\b}\), there exists a constant \(C_{T,\phi}>0\), independent of \(\eps\), such that
\begin{equation*}
\sup_{0\leq t\leq T}\left|\E\phi(X_t^\eps,Y_t^\eps)-\E\phi(\bar X_t,\bar Y_t)\right|
\leq C_{T,\phi}\,\eps^{\beta/2},
\end{equation*}
where \((\bar X_t,\bar Y_t)\) is the solution to \eqref{eq:limit}.
\end{proposition}

\begin{proof}
Fix \(T>0\), let \(u\) be the solution to \eqref{eq:cau} and define
\[
\tilde u(t,x,y):=u(T-t,x,y).
\]
Then it is obvious that
\begin{align*}
&\tilde u(T,x,y)=u(0,x,y)=\phi(x,y),\\
&\tilde u(0,x,y)=u(T,x,y)=\E\phi(\bar X_T(x),\bar Y_T(x,y)).
\end{align*}
Applying It\^{o}'s formula and using \eqref{eq:cau}, we deduce that
\begin{align}\label{eq:xy1}
&\E\phi(X_{T}^\eps,Y_{T}^\eps)-\E\phi(\bar X_{T},\bar Y_T)=\E\tilde u(T,X_T^\eps,Y_T^\eps)-\tilde u(0,x,0)\no\\
&=\E\left(\frac{1}{\sqrt{\eps}}\int_0^T\big\langle \nabla_x\tilde u(t,X_t^\eps,Y_t^\eps), V_t^\eps\big\rangle\dif t\right)+\E\left(\frac{1}{\sqrt{\eps}}\int_0^T\big\langle \nabla_y\tilde u(t,X_t^\eps,Y_t^\eps),h(X^\eps_t,V^\eps_t)\big\rangle\dif t\right)\no\\
&\quad+\E\left(\int_0^T\left\langle \nabla_y\tilde u(t,X_t^\eps,Y_t^\eps),\frac{1}{\sqrt{\eps}}\big[B(\bar X_t)-B(X^\eps_t)\big]-\nabla B(X_t^\eps)Y_t^\eps\right\rangle\dif t\right)\no\\
&\quad+\frac{1}{2}\E\bigg(\int_0^T\nabla^2_{y}\tilde u(t,X_t^\eps,Y_t^\eps):\Big[\frac{1}{\eps}\big(a(\bar X_t)-a(X_t^\eps)\big)\big(a(\bar X_t)-a(X_t^\eps)\big)^*\no\\
&\qquad\qquad\qquad-\nabla a(X^\eps_t)[Y^\eps_t]\big(\nabla a(X^\eps_t)[Y^\eps_t]\big)^*-\cQ(X_t^\eps)\Big]\dif t\bigg)\no\\
&\quad-\E\left(\int_0^T\big\langle \nabla_x\tilde u(t,X_t^\eps,Y_t^\eps),B(X_t^\eps)\big\rangle\dif t\right)-\frac{1}{2}\E\left(\int_0^T\nabla^2_{x}\tilde u(t,X_t^\eps,Y_t^\eps):\big[a(X_t^\eps)a^*(X_t^\eps)\big]\dif t\right)\no\\
&\quad-\E\left(\int_0^T\nabla^2_{xy}\tilde u(t,X_t^\eps,Y_t^\eps):\big[a(X^\eps_t)\big(\nabla a(X^\eps_t)[Y^\eps_t]\big)^*\big]\dif t\right).
\end{align}
Define
 \[
 f_u(t,x,y,v):=\langle \nabla_x\tilde u(t,x,y),v\rangle+\langle \nabla_y\tilde u(t,x,y),h(x,v)\rangle.
 \]
By Lemma~\ref{lem:regu}, we have \(f_u\in C_p^{(1+\b)/2,1+\b,1+\b,\d}\). Meanwhile, by the definition \eqref{eq:defh} of \(h\), we can check that \(f_u\) satisfies the centering condition \eqref{eq:cenf}, i.e.,
\begin{equation*}
\int_{\R^d}f_u(t,x,y,v)\,\zeta^x(\dif v)=0.
\end{equation*}
Moreover, the solution of the associated  Poisson equation
\begin{equation*}
\sL_0(x)\Psi_u(t,x,y,v)=-f_u(t,x,y,v)
\end{equation*}
is given by
\begin{equation*}
\Psi_u(t,x,y,v)=
\big\langle \nabla_x\tilde u(t,x,y),\Gamma(x,v)\big\rangle+\big\langle \nabla_y\tilde u(t,x,y),\Phi(x,v)\big\rangle,
\end{equation*}
where $\Phi(x,v)$ satisfies (\ref{poissh}).
Applying estimate \eqref{eq:est2} in Proposition~\ref{prop:flu} to $f_u$, we obtain
\begin{equation}\label{eq:xy2}
\left|\E\left(\frac{1}{\sqrt{\eps}}\int_0^Tf_u(t,X_t^\eps,Y_t^\eps,V_t^\eps)\dif t\bigg)
-\E\bigg(\int_0^T\bar{\cH}_u(t,X_t^\eps,Y_t^\eps)\dif t\right)\right|
\leq C_T\,\eps^{\b/2},
\end{equation}
where \(\bar{\cH}_u\) is given by
\begin{equation*}
\bar{\cH}_u(t,x,y):=\int_{\R^d}\cH_u(t,x,y,v)\,\zeta^x(\dif v),
\end{equation*}
with
\begin{align}\label{eq:defhu}
\cH_u(t,x,y,v)&:=\big\langle \nabla_x\Psi_u(t,x,y,v),v\big\rangle+\big\langle
\nabla_v\Psi_u(t,x,y,v),F(x)\big\rangle\no\\
&\quad+\big\langle \nabla_y\Psi_u(t,x,y,v),h(x,v)\big\rangle\no\\
&\quad+\nabla_{yv}^2\Psi_u(t,x,y,v):\big(\nabla a(x)[y+\Gamma(x,v)]\sigma^*(x)\big).
\end{align}
For simplicity, we state the following claim:

\smallskip
\noindent{\bf Claim:} Let $\cH_u$ be defined by (\ref{eq:defhu}), then its average $\bar{\cH}_u$ is given by
\begin{align}\label{eq:defhubar}
&\bar{\cH}_u(t,x,y)=
\big\langle \nabla_x\tilde u(t,x,y),B(x)\big\rangle+\frac12\nabla_{x}^2\tilde u(t,x,y):\big(a(x)a^*(x)\big)\no\\
&\quad+\nabla_{xy}^2\tilde u(t,x,y):\big[a(x)\big(\nabla a(x)[y]\big)^*\big]+\frac12 \nabla_{y}^2\tilde u(t,x,y):\big(\cQ(x)-\cQ_0(x)\big),
\end{align}
where \(\cQ_0(x)\) is defined by
\begin{equation}\label{eq:defQ0}
\cQ_0(x):=\int_{\R^d}\nabla a(x)[\Gamma(x,v)]\big(\nabla a(x)[\Gamma(x,v)]\big)^*\,\zeta^x(\dif v).
\end{equation}
Once this claim is proven, substituting \eqref{eq:xy2} and \eqref{eq:defhubar} into \eqref{eq:xy1}, and using the resulting cancellations, we have
\begin{equation*}
\big|\E\phi(X_{T}^\eps,Y_{T}^\eps)-\E\phi(\bar X_{T},\bar Y_T)\big|\leq C_T\,\eps^{\b/2}+\cU_1(\eps)+\cU_2(\eps),
\end{equation*}
where
\begin{align*}
&\cU_1(\eps):=\left|\E\left(\int_0^T\left\langle \nabla_y\tilde u(t,X_t^\eps,Y_t^\eps),\frac{1}{\sqrt{\eps}}\big[B(\bar X_t)-B(X^\eps_t)\big]-\nabla B(X_t^\eps)Y_t^\eps\right\rangle\dif t\right)\right|,\\
&\cU_2(\eps):=\frac{1}{2}\bigg|\E\bigg(\int_0^T\nabla^2_{y}\tilde u(t,X_t^\eps,Y_t^\eps):\Big[\frac{1}{\eps}\big(a(\bar X_t)-a(X_t^\eps)\big)\big(a(\bar X_t)-a(X_t^\eps)\big)^*\no\\
&\qquad\qquad\qquad-\nabla a(X^\eps_t)[Y^\eps_t]\big(\nabla a(X^\eps_t)[Y^\eps_t]\big)^*
-\cQ_0(X_t^\eps)\Big]\dif t\bigg)\bigg|.
\end{align*}
For the first term, we deduce that
\begin{align*}
\cU_1(\eps)&\leq\bigg|\E\bigg(\int_0^T\Big\langle \nabla_y\tilde u(t,X_t^\eps,Y_t^\eps),\frac{1}{\sqrt{\eps}}\big[B(\bar X_t)-B(X^\eps_t)\big]-\nabla B(X_t^\eps)Y_t^\eps\\
&\qquad\qquad\quad-\nabla B(X_t^\eps)\Gamma(X_t^\eps,V_t^\eps)\Big\rangle\dif t\bigg)\bigg|\\
&\quad+\left|\E\left(\int_0^T\left\langle \nabla_y\tilde u(t,X_t^\eps,Y_t^\eps),\nabla B(X_t^\eps)\Gamma(X_t^\eps,V_t^\eps)\right\rangle\dif t\right)\right|\\
&=:\cU_{1,1}(\eps)+\cU_{1,2}(\eps).
\end{align*}
Since \(B\in C_b^{1+\b}\), by estimate \eqref{eq:rem2}, Lemma~\ref{lem:regu} and the uniform moment estimates, we obtain
\begin{equation}\label{eq:xy3}
\cU_{1,1}(\eps)\leq C_1\,\eps^{\beta/2}\int_0^T\big(1+\E\norm{Y_t^\eps}^p+\E\norm{Z_t^\eps}^p\big)
\dif t\leq C_T\,\eps^{\beta/2}.
\end{equation}
Furthermore,
\[
(t,x,y,v)\longmapsto\left\langle \nabla_y\tilde u(t,x,y),\nabla B(x)\Gamma(x,v)\right\rangle
\]
belongs to \(C_p^{\beta/2,\beta,\beta,\delta}\). Since \(\int_{\R^d}\Gamma(x,v)\,\zeta^x(\dif v)=0\), using estimate \eqref{eq:est1} in Proposition~\ref{prop:flu}  we can get
\begin{equation*}
\cU_{1,2}(\eps)\leq C_T\,\eps^{\beta/2},
\end{equation*}
which together with \eqref{eq:xy3} yields that
\begin{equation*}
\cU_1(\eps)\leq C_T\,\eps^{\beta/2}.
\end{equation*}
We next estimate \(\cU_2(\eps)\). Recall from \eqref{eq:alin} that
\begin{equation*}
\frac{1}{\sqrt{\eps}}\left(a(\bar X_t)-a(X_t^\eps)\right)
=\nabla a(X_t^\eps)\left[Y_t^\eps+\Gamma(X_t^\eps,V_t^\eps)\right]+r_t^\eps,
\end{equation*}
where
\begin{equation}\label{eq:reps}
\big|r_t^\eps\big|\leq C_2\,\eps^{\beta/2}\big(1+\norm{Z_t^\eps}^{1+\beta}\big).
\end{equation}
Then, we  deduce that
\begin{align*}
&\frac{1}{\eps}\big(a(\bar X_t)-a(X_t^\eps)\big)\big(a(\bar X_t)-a(X_t^\eps)\big)^*\\
&=\nabla a(X_t^\eps)[Y_t^\eps]\big(\nabla a(X_t^\eps)[Y_t^\eps]\big)^*
+\nabla a(X_t^\eps)[Y_t^\eps]\big(\nabla a(X_t^\eps)[\Gamma(X_t^\eps,V_t^\eps)]\big)^*\\
&\quad+\nabla a(X_t^\eps)[\Gamma(X_t^\eps,V_t^\eps)]\big(\nabla a(X_t^\eps)[\Gamma(X_t^\eps,V_t^\eps)]\big)^*
+\nabla a(X_t^\eps)[\Gamma(X_t^\eps,V_t^\eps)]\big(\nabla a(X_t^\eps)[Y_t^\eps]\big)^*\\
&\quad+\nabla a(X_t^\eps)\left[Y_t^\eps+\Gamma(X_t^\eps,V_t^\eps)\right]\big(r_t^\eps\big)^*
+r_t^\eps\big(\nabla a(X_t^\eps)[Y_t^\eps+\Gamma(X_t^\eps,V_t^\eps)]\big)^*+
r_t^\eps(r_t^\eps)^*.
\end{align*}
Substituting this into \(\cU_2(\eps)\), we have
\begin{align}\label{eq:xy4}
\cU_2(\eps)&\leq\frac{1}{2}\left|\E\left(\int_0^T\nabla^2_{y}\tilde u(t,X_t^\eps,Y_t^\eps):\left[\nabla a(X^\eps_t)[Y^\eps_t]
\big(\nabla a(X^\eps_t)[\Gamma(X^\eps_t,V_t^\eps)]\big)^*\right]\right)\right|\no\\
&\quad+\frac{1}{2}\left|\E\left(\int_0^T\nabla^2_{y}\tilde u(t,X_t^\eps,Y_t^\eps):\left[\nabla a(X^\eps_t)[\Gamma(X_t^\eps,V^\eps_t)]
\big(\nabla a(X^\eps_t)[Y_t^\eps]\big)^*\right]\right)\right|\no\\
&\quad+\frac{1}{2}\bigg|\E\bigg(\int_0^T\nabla^2_{y}\tilde u(t,X_t^\eps,Y_t^\eps):\big[\nabla a(X^\eps_t)[\Gamma(X_t^\eps,V^\eps_t)]
\big(\nabla a(X^\eps_t)[\Gamma(X^\eps_t,V_t^\eps)]\big)^*\no\\
&\qquad\qquad\quad-\cQ_0(X_t^\eps)\big]\dif t\bigg)\bigg|\no\\
&\quad+\frac{1}{2}\bigg|\E\bigg(\int_0^T\nabla^2_{y}\tilde u(t,X_t^\eps,Y_t^\eps):
\big[\nabla a(X_t^\eps)[Y_t^\eps+\Gamma(X_t^\eps,V_t^\eps)]\big(r_t^\eps\big)^*\no\\
&\qquad\qquad\quad+r_t^\eps\big(\nabla a(X_t^\eps)[Y_t^\eps+\Gamma(X_t^\eps,V_t^\eps)]\big)^*+r_t^\eps(r_t^\eps)^*\big]\dif t\bigg)\bigg|.
\end{align}
The first three terms on the right-hand side of \eqref{eq:xy4} are centered with respect to
\(\zeta^x\). Indeed,
\[
\int_{\R^d}\Gamma(x,v)\,\zeta^x(\dif v)=0
\]
and  by \eqref{eq:defQ0},
\[
\int_{\R^d}\nabla a(x)[\Gamma(x,v)]\bigl(\nabla a(x)[\Gamma(x,v)]\bigr)^*\,\zeta^x(\dif v)=\cQ_0(x).
\]
Hence, applying \eqref{eq:est1} in Proposition~\ref{prop:flu}, we obtain that the first three terms on the right-hand side of \eqref{eq:xy4} are bounded by \(C_T\,\eps^{\b/2}\).
At the same time,  by \eqref{eq:reps} we have
\begin{align*}
&\frac{1}{2}\bigg|\E\bigg(\int_0^T\nabla^2_{y}\tilde u(t,X_t^\eps,Y_t^\eps):
\big[\nabla a(X_t^\eps)[Y_t^\eps+\Gamma(X_t^\eps,V_t^\eps)]\big(r_t^\eps\big)^*\\
&\qquad\qquad\quad+r_t^\eps\big(\nabla a(X_t^\eps)[Y_t^\eps+\Gamma(X_t^\eps,V_t^\eps)]\big)^*+r_t^\eps(r_t^\eps)^*\big]\dif t\bigg)\bigg|\\
&\leq C_T\,\eps^{\b/2}\int_0^T\big(1+\E\norm{V_t^\eps}^p+\E\norm{Y_t^\eps}^p+\E\norm{Z_t^\eps}^p\big)\dif t\leq C_T\,\eps^{\b/2}.
\end{align*}
This together with \eqref{eq:xy4} implies that
\begin{equation*}
\cU_2(\eps)\leq C_T\,\eps^{\beta/2}.
\end{equation*}
Combining the estimates of \(\cU_1(\eps)\) and \(\cU_2(\eps)\), we arrive at
\begin{equation*}
\big|\E\phi(X_{T}^\eps,Y_{T}^\eps)-\E\phi(\bar X_{T},\bar Y_T)\big|\leq C_T\,\eps^{\beta/2}.
\end{equation*}
The proof is finished.
\end{proof}

\subsection{Computation of the averaged function \(\bar{\cH}_u\)}

We now identify the averaged function \(\bar{\cH}_u\) arising from (\ref{eq:xy2}).

Recall that
\begin{align*}
&\Psi_u(t,x,y,v)=\p_{x_i}\tilde u(t,x,y)\Gamma_i(x,v)+\p_{y_i}\tilde u(t,x,y)\Phi_i(x,v), \\
&\Gamma_i(x,v)=(\gamma^{-1})_{i\ell}(x)v_\ell,\qquad \p_{v_k}\Phi_i(x,v)=-2\big(P_i(x)\big)_{k\ell}v_\ell,
\end{align*}
where, for simplicity, repeated indices are summed from \(1\) to \(d\).
Below, we give the proof of the claim (\ref{eq:defhubar}).

\begin{proof}[{\bf Proof of formula \eqref{eq:defhubar}}]
By \eqref{eq:defhu}, a direct expansion of \(\cH_u\) gives
\begin{align}\label{eq:defhu1}
\cH_u(t,x,y,v)&=\p_{x_ix_j}^2\tilde u(t,x,y)\Gamma_i(x,v)v_j+\p_{x_i}\tilde u(t,x,y)\p_{x_j}\Gamma_i(x,v)v_j\no\\
&\quad+\p_{x_jy_i}^2\tilde u(t,x,y)\Phi_i(x,v)v_j+\p_{y_i}\tilde u(t,x,y)\p_{x_j}\Phi_i(x,v)v_j\no\\
&\quad+\p_{x_i}\tilde u(t,x,y)(\g^{-1})_{ik}(x)F_k(x)+
\p_{y_i}\tilde u(t,x,y)\p_{v_k}\Phi_i(x,v)F_k(x)\no\\
&\quad+\p_{x_iy_r}^2\tilde u(t,x,y)\Gamma_i(x,v)h_r(x,v)+\p_{y_iy_r}^2\tilde u(t,x,y)\Phi_i(x,v)h_r(x,v)\no\\
&\quad+\p_{y_iy_r}^2\tilde u(t,x,y)\p_{v_k}\Phi_i(x,v)\cD_{rk}(x,y,v)\no\\
&\quad+\p_{x_iy_r}^2\tilde u(t,x,y)(\g^{-1})_{ik}(x)\cD_{rk}(x,y,v)\no\\
&=:\sum_{\ell=1}^{10}\cH_u^\ell(t,x,y,v),
\end{align}
where
\[
\cD_{rk}(x,y,v)=\p_{x_q}a_{r\ell}(x)\big(y_q+\Gamma_q(x,v)\big)\sigma_{k\ell}(x).
\]
We compute the average of the terms in \eqref{eq:defhu1} in four
steps.

\medskip
\noindent
\emph{Step 1: Terms with zero average.}
Since \(\zeta^x=\mathcal N(0,M(x))\) is a centered Gaussian measure,
odd functions of \(v\) have zero average. In particular,
\begin{align*}
&\int_{\R^d}\Phi_i(x,v)v_j\,\zeta^x(\dif v)=0,\qquad\quad\,\,\,\,
\int_{\R^d}\p_{x_j}\Phi_i(x,v)v_j\,\zeta^x(\dif v)=0,\\
&\int_{\R^d}\p_{v_k}\Phi_i(x,v)F_k(x)\,\zeta^x(\dif v)=0,\quad
\int_{\R^d}\Gamma_i(x,v)h_r(x,v)\,\zeta^x(\dif v)=0.
\end{align*}
As a result, we have
\begin{equation}\label{eq:defhu2}
\int_{\R^d} \bigl(\cH_u^3+\cH_u^4+\cH_u^6+\cH_u^7\bigr)(t,x,y,v)\,\zeta^x(\dif v)=0.
\end{equation}

\medskip
\noindent
\emph{Step 2: Terms involving derivatives in \(x\).}
Since \(\int_{\R^d}v_jv_\ell\,\zeta^x(\dif v)=M_{j\ell}(x)\), we have
\begin{equation*}
\int_{\R^d}\Gamma_i(x,v)v_j\,\zeta^x(\dif v)=(\gamma^{-1})_{i\ell}(x)\int_{\R^d}v_\ell v_j\,\zeta^x(\dif v)=(\gamma^{-1})_{i\ell}(x)M_{\ell j}(x),
\end{equation*}
which in turn implies that
\begin{align*}
\int_{\R^d}\cH_u^1(t,x,y,v)\zeta^x(\dif v)&=\p_{x_ix_j}^2\tilde u(t,x,y)(\g^{-1})_{i\ell}(x)M_{j\ell}(x)\\
&=\nabla_{x}^2\tilde u(t,x,y):\big(\g^{-1}(x)M(x)\big).
\end{align*}
By the symmetry of \(\nabla_x^2 \tilde{u}\), we obtain
\begin{equation*}
\nabla_{x}^2\tilde u:\big(\g^{-1}M\big)=
\frac12\nabla_{x}^2\tilde u:\big[\g^{-1}M+M\g^{-*}\big]=\frac12\nabla_{x}^2\tilde u:\big(aa^*\big),
\end{equation*}
where the last equality follows from the Lyapunov identity (\ref{eq:lyapM}). Indeed, multiplying the identity from the left by \(\g^{-1}(x)\) and from the right by
\(\g^{-*}(x)\), we have
\[
\g^{-1}(x)M(x)+M(x)\gamma^{-*}(x)=a(x)a^*(x).
\]
Thus, it holds that
\begin{equation}\label{eq:defhu3}
\int_{\R^d}\cH_u^1(t,x,y,v)\,\zeta^x(\dif v)=\frac12\nabla_{x}^2\tilde u(t,x,y):\big(a(x)a^*(x)\big).
\end{equation}
Next, differentiating \(\Gamma_i(x,v)\) with respect to \(x\), we obtain
\begin{align*}
\int_{\R^d}\p_{x_j}\Gamma_i(x,v)v_j\,\zeta^x(\dif v)
&=\p_{x_j}(\gamma^{-1})_{i\ell}(x)\int_{\R^d}v_\ell v_j\,\zeta^x(\dif v)\\
&=\p_{x_j}(\gamma^{-1})_{i\ell}(x)M_{\ell j}(x).
\end{align*}
This together with the definition of \(B\) implies that
\begin{align}\label{eq:defhu4}
&\int_{\R^d}\bigl(\cH_u^2+\cH_u^5\bigr)(t,x,y,v)\,\zeta^x(\dif v)\no\\
&=\p_{x_i}\tilde u(t,x,y)\left[\p_{x_j}(\gamma^{-1})_{i\ell}(x)M_{\ell j}(x)
    +(\gamma^{-1})_{ik}(x)F_k(x)\right]\no\\
&=\big\langle \nabla_x\tilde u(t,x,y),B(x)\big\rangle.
\end{align}

\medskip
\noindent
\emph{Step 3: Terms involving mixed derivatives.}
By the centering property of \(\Gamma\), it holds that
\begin{align*}
\int_{\R^d}\cD_{rk}(x,y,v)\,\zeta^x(\dif v)&=\p_{x_q}a_{r\ell}(x)\sigma_{k\ell}(x)\int_{\R^d}\big(y_q+\Gamma_q(x,v)\big)\,\zeta^x(\dif v)\\
&=\p_{x_q}a_{r\ell}(x)y_q\sigma_{k\ell}(x).
\end{align*}
As a result, we have
\begin{align}\label{eq:defhu5}
\int_{\R^d}\cH_u^{10}(t,x,y,v)\,\zeta^x(\dif v)&=\p_{x_iy_r}^2\tilde u(t,x,y)(\g^{-1})_{ik}(x)\p_{x_q}a_{r\ell}(x)y_q\sigma_{k\ell}(x)\no\\
&=\nabla_{xy}^2\tilde u:\big[a(x)\big(\nabla a(x)[y]\big)^*\big].
\end{align}

\medskip
\noindent
\emph{Step 4: Terms involving second derivatives in \(y\).}
Set
\[
\cA(x,y,v):=\nabla a(x)[y+\Gamma(x,v)],\qquad
\sU(x,v):=\nabla_v\Phi(x,v)\sigma(x).
\]
Since \(\cD(x,y,v)=\cA(x,y,v)\sigma^*(x)\), the last two terms involving \(\nabla_y^2\tilde u\) can be written as
\begin{align}\label{eq:uy0}
&\cH_u^8(t,x,y,v)+\cH_u^9(t,x,y,v)\no\\
&=\nabla_y^2\tilde u(t,x,y):\big[\Phi(x,v)h^*(x,v)+\sU(x,v)\cA^*(x,y,v)\big].
\end{align}
Using the symmetry of \(\nabla_y^2\tilde u\), we obtain
\begin{equation}\label{eq:uy1}
\begin{aligned}
&\nabla_{y}^2\tilde u:\big[\Phi h^*\big]
=\frac12\nabla_{y}^2\tilde u:\big[\Phi h^*+h\Phi^*\big],\\
&\nabla_{y}^2\tilde u:\big[\sU\cA^*\big]
=\frac12\nabla_{y}^2\tilde u:\big[\sU\cA^*+\cA\sU^*\big].
\end{aligned}
\end{equation}
We next use the Poisson equation (\ref{poissh}). For every \(1\leq i,r\leq d\), the product formula gives
\begin{align}\label{eq:uy2}
&\sL_0(x)\big(\Phi_i(x,v)\Phi_r(x,v)\big)\no\\
&=\Phi_i(x,v)\sL_0(x)\Phi_r(x,v)+\Phi_r(x,v)\sL_0(x)\Phi_i(x,v)
+\big\langle\sigma^*(x)\nabla_v\Phi_i(x,v),\sigma^*(x)\nabla_v\Phi_r(x,v)\big\rangle\no\\
&=-\Phi_i(x,v)h_r(x,v)-\Phi_r(x,v)h_i(x,v)+\big(\sU(x,v)\sU^*(x,v)\big)_{ir}.
\end{align}
Since \(\zeta^x\) is the invariant measure under the operator \(\sL_0(x)\), we have
\[
\int_{\R^d}\sL_0(x)\big(\Phi_i(x,v)\Phi_r(x,v)\big)\,\zeta^x(\dif v)=0,
\]
Integrating \eqref{eq:uy2} with respect to \(\zeta^x\), we obtain the
 identity
\begin{align}\label{eq:uy3}
\int_{\R^d}\left[\Phi(x,v)h^*(x,v)+h(x,v)\Phi^*(x,v)\right]\,\zeta^x(\dif v)=\int_{\R^d} \sU(x,v)\sU^*(x,v)\,\zeta^x(\dif v).
\end{align}
Combining \eqref{eq:uy0}, \eqref{eq:uy1} and \eqref{eq:uy3}, we arrive at
\begin{align*}
&\int_{\R^d}\big(\cH_u^8+\cH_u^9\big)(t,x,y,v)\,\zeta^x(\dif v)\no\\
&=\frac{1}{2}\nabla_{y}^2\tilde u(t,x,y):\int_{\R^d}\big[\sU(x,v)\sU^*(x,v)+\sU(x,v)\cA^*(x,y,v)\\
&\qquad\qquad\qquad\qquad\qquad+\cA(x,y,v)\sU^*(x,v)\big]\,\zeta^x(\dif v).
\end{align*}
Using the identity
\[
\sU\sU^*+\sU\cA^*+\cA\sU^*=(\cA+\sU)(\cA+\sU)^*-\cA\cA^*,
\]
we further have
\begin{align*}
&\int_{\R^d}\big(\cH_u^8+\cH_u^9\big)(t,x,y,v)\,\zeta^x(\dif v)\\
&=\frac{1}{2}\nabla_{y}^2\tilde u(t,x,y):\int_{\R^d}\Big[\big(\cA(x,y,v)+\sU(x,v)\big)\big(\cA(x,y,v)+\sU(x,v)\big)^*\\
&\qquad\qquad\qquad\qquad\qquad-\cA(x,y,v)\cA^*(x,y,v)\Big]\,\zeta^x(\dif v).
\end{align*}
By the
definition (\ref{defr1}) of \(\cR\), we have
\[
\cA(x,y,v)+\sU(x,v)=\nabla a(x)[y]+\cR(x)v.
\]
Since \(\zeta^x\) is centered, it follows from the definitions of
\(\cQ\) and \(\cQ_0\) that
\begin{align*}
&\int_{\R^d}(\cA(x,y,v)+\sU(x,v))(\cA(x,y,v)+\sU(x,v))^*\,\zeta^x(\dif v)\\
&=\nabla a(x)[y]\bigl(\nabla a(x)[y]\bigr)^*+\cQ(x),
\end{align*}
whereas
\begin{equation*}
\int_{\R^d} \cA(x,y,v)\cA^*(x,y,v)\,\zeta^x(\dif v)=\nabla a(x)[y]\bigl(\nabla a(x)[y]\bigr)^*+\cQ_0(x).
\end{equation*}
Therefore, we get
\begin{equation*}
\int_{\R^d}\big(\cH_u^8+\cH_u^9\big)(t,x,y,v)\,\zeta^x(\dif v)=\frac12 \nabla_{y}^2\tilde u(t,x,y):\big(\cQ(x)-\cQ_0(x)\big),
\end{equation*}
which together with \eqref{eq:defhu1}-\eqref{eq:defhu5} proves \eqref{eq:defhubar}. The proof is finished.
\end{proof}

\begin{remark}
The identity \eqref{eq:uy3} combines the martingale
term generated by the quadratic corrector with the fast component of
the common-noise linearization. It is essential to combine these two
contributions before averaging. Treating them separately would omit
the corresponding cross-variation terms and can not
yield the correct covariance matrix \(\cQ\).
\end{remark}

\section{Conditional equilibrium of the fast component}
\label{sec:fast}

We first study the frozen Ornstein--Uhlenbeck process. For \(x,y,z\in\R^d\), let \(Z_t^{x,y}(z)\) denote the solution of \eqref{eq:forz} with initial value \(z\), namely,
\begin{equation}\label{eq:forz1}
\dif Z_t^{x,y}(z)=-\g(x)\big(Z_t^{x,y}(z)-y\big)\dif t
+a(x)\dif W_t, \qquad  Z_0^{x,y}(z)=z.
\end{equation}
Its explicit representation is given by
\begin{equation}\label{eq:ouexp}
Z_t^{x,y}(z)=y+\mathrm e^{-t\gamma(x)}(z-y)+\int_0^t\mathrm e^{-(t-r)\gamma(x)}a(x)\dif W_r.
\end{equation}
We denote the corresponding Markov semigroup by
\begin{equation}\label{sem}
\cP_t^{x,y}\varphi(z):=\E\varphi(Z_t^{x,y}(z)).
\end{equation}
Given a function \(\varphi(z)\), define
\begin{equation}\label{barphi}
\bar\varphi(x,y):=\int_{\R^d}\varphi(z)\,\mu^{x,y}(\dif z),
\end{equation}
where \(\mu^{x,y}=\mathcal N\bigl(y,N(x)\bigr)\) is the invariant measure of \eqref{eq:forz1}.

\vspace{1mm}
The following lemma gives the exponential convergence of the frozen
Ornstein--Uhlenbeck semigroup as well as
the derivative estimates
needed below.

\begin{lemma}\label{lem:estou}
Assume {\bf(A1)} and {\bf(A2)} hold, and let
\(\varphi\in C_b^{\a}\) with $\a\in(0,1)$. Define
\begin{equation}\label{ww}
w(t,z;x,y):=\cP_t^{x,y}\varphi(z)-\bar\varphi(x,y).
\end{equation}
where $\bar\varphi$ is given by \eqref{barphi}.
Then there exist constants $C_\alpha,\kappa>0$  such that, for every $t>0$ and
$x,y,z\in\R^d$,
\begin{align}
&|w(t,z;x,y)|
+|\nabla_xw(t,z;x,y)|
+|\nabla_yw(t,z;x,y)|\no\\
&\qquad\quad\,\qquad+|\nabla_y^2w(t,z;x,y)|
+|\nabla_{yz}^2w(t,z;x,y)|
\leq
C_\alpha\,
\mathrm e^{-\kappa t}(1+|z-y|),
\label{eq:estou}
\\
&|\nabla_zw(t,z;x,y)|
\leq
C_\alpha
(t\wedge1)^{-(1-\alpha)/2}\,\mathrm e^{-\kappa t},
\label{eq:estou-z}
\\
&|\nabla_z^2w(t,z;x,y)|
\leq
C_\alpha(t\wedge1)^{-1+\alpha/2}\,
\mathrm e^{-\kappa t}.
\label{eq:estou-zz}
\end{align}
Moreover, we have
\(
\bar\varphi\in C_b^{2+\beta,2+\beta}\).
\end{lemma}

\begin{proof}
For simplicity, set
\[
E_t(x):=\mathrm e^{-t\gamma(x)},
\qquad
N_t(x):=\int_0^tE_r(x)a(x)a^*(x)E_r^*(x)\,\dif r,
\qquad
m_t:=y+E_t(x)(z-y).
\]
For a symmetric positive definite matrix $A$, let $p_A$
denote the density of $\cN(0,A)$. Then, by \eqref{eq:ouexp} we have
\begin{equation}\label{eq:ou-gaussian-representation}
\cP_t^{x,y}\varphi(z)
=
\int_{\R^d}\varphi(\eta)p_{N_t(x)}(\eta-m_t)\,\dif\eta,
\qquad
\bar\varphi(x,y)
=
\int_{\R^d}\varphi(\eta)p_{N(x)}(\eta-y)\,\dif\eta.
\end{equation}
According to Lemma~\ref{lem:reg}, $N\in C_b^{2+\beta}$ is uniformly
positive definite and bounded. Since the spatial and parameter
derivatives of $p_A$ have uniformly bounded $L^1$-norms, we obtain
\begin{equation}\label{eq:ou-barphi-regularity}
\bar\varphi\in C_b^{2+\beta,2+\beta}\quad\text{with}\quad
\|\bar\varphi\|_{C_b^{2+\beta,2+\beta}}
\leq C_0\|\varphi\|_\infty.
\end{equation}
Below, we prove the derivative estimates, and we treat the small
and long time estimates separately.

\medskip
\noindent
\emph{Step 1: Estimates for $0<t\leq1$.}
By the assumptions on the coefficients and Duhamel's formula, we have
\begin{equation*}
C_1^{-1}t\mI_d\leq N_t(x)\leq C_1t\mI_d,
\qquad
|\mI_d-E_t(x)|+|\nabla_xE_t(x)|+|\nabla_xN_t(x)|
\leq C_1t.
\end{equation*}
For $k=1,2$, by standard computation we deduce that
\begin{align*}
\left|
\int_{\R^d}\varphi(\eta)
\nabla_\eta^kp_{N_t(x)}(\eta-m)\,\dif\eta
\right|
&=
\left|
\int_{\R^d}
[\varphi(m+\eta)-\varphi(m)]
\nabla_\eta^kp_{N_t(x)}(\eta)\,\dif\eta
\right|\\
&\leq
C_2
\int_{\R^d}|\eta|^\alpha
|\nabla_\eta^kp_{N_t(x)}(\eta)|\,\dif\eta\leq C_2 t^{(\alpha-k)/2}.
\end{align*}
Since $\nabla_zm_t=E_t(x)$ and
$\nabla_ym_t=\mI_d-E_t(x)$, it follows that
\begin{align*}
|\nabla_z\cP_t^{x,y}\varphi(z)|
&\leq C_3t^{-(1-\alpha)/2},
&
|\nabla_z^2\cP_t^{x,y}\varphi(z)|
&\leq C_3t^{-1+\alpha/2},\\
|\nabla_y\cP_t^{x,y}\varphi(z)|
&\leq C_3t^{(1+\alpha)/2},
&
|\nabla_y^2\cP_t^{x,y}\varphi(z)|
&\leq C_3t^{1+\alpha/2},
\end{align*}
and
\[
|\nabla_{yz}^2\cP_t^{x,y}\varphi(z)|
\leq C_3t^{\alpha/2}.
\]
For the $x$-derivative,  differentiating the mean and covariance in
\eqref{eq:ou-gaussian-representation}, we obtain
\begin{align*}
|\nabla_x\cP_t^{x,y}\varphi(z)|
&\leq C_4
\left(
|\nabla_xE_t(x)|\,|z-y|\,t^{(\alpha-1)/2}
+
|\nabla_xN_t(x)|\,t^{-1+\alpha/2}
\right)\\
&\leq C_4
\left(t^{(1+\alpha)/2}|z-y|+t^{\alpha/2}\right).
\end{align*}
Together with \eqref{eq:ou-barphi-regularity} and
$|w|\leq2\|\varphi\|_\infty$, we obtain the
required bounds for $0<t\leq1$.

\medskip
\noindent
\emph{Step 2: Estimates for $t\geq1$.}
By the coercivity of $\gamma$ and Duhamel's formula, we have
\[
|E_t(x)|\leq\mathrm e^{-\lambda_\gamma t},
\qquad
|\nabla_xE_t(x)|\leq C_5t\mathrm e^{-\lambda_\gamma t}.
\]
Note that
\[
N(x)-N_t(x)
=
\int_t^{+\infty} E_r(x)a(x)a^*(x)E_r^*(x)\,\dif r.
\]
We deduce that  for some $\kappa>0$,
\begin{align*}
|E_t(x)|+|\nabla_xE_t(x)|
+|N_t(x)-N(x)|
+|\nabla_xN_t(x)-\nabla_xN(x)|
\leq C_5\mathrm e^{-\kappa t}.
\end{align*}
The matrices $N_t(x)$, $t\geq1$, and $N(x)$ are uniformly
positive definite and bounded. Hence, for $k=0,1,2$, by
the $L^1$-bounds for Gaussian derivatives we deduce that
\begin{align*}
&\big\|
\nabla^kp_{N_t(x)}(\cdot-m_t)
-\nabla^kp_{N(x)}(\cdot-y)
\big\|_{L^1}\\
&\qquad\leq
C_6\big(|N_t(x)-N(x)|+|E_t(x)(z-y)|\big)
\leq C_6\mathrm e^{-\kappa t}(1+|z-y|).
\end{align*}
Differentiating \eqref{eq:ou-gaussian-representation}
and using $\nabla_ym_t=\mI_d-E_t(x)$, we further obtain
\[
|w|+|\nabla_yw|+|\nabla_y^2w|
\leq C_7
\mathrm e^{-\kappa t}(1+|z-y|).
\]
Similarly, differentiation in $x$-variable we have
\begin{align*}
|\nabla_xw|
&\leq C_8
\Big(
|\nabla_xE_t(x)|\,|z-y|
+|\nabla_xN_t(x)-\nabla_xN(x)|\\
&\hspace{38mm}
+|N_t(x)-N(x)|+|E_t(x)(z-y)|
\Big)\\
&\leq C_8
\mathrm e^{-\kappa t}(1+|z-y|).
\end{align*}
Finally, each derivative in $z$ contributes a factor $E_t(x)$,
whereas $\bar\varphi$ is independent of $z$. Thus, we have
\[
|\nabla_zw|+|\nabla_z^2w|+|\nabla_{yz}^2w|
\leq C_9\mathrm e^{-\kappa t}.
\]
Combining the two steps we obtain the desired estimates.
The proof is finished.
\end{proof}

We now return to \(Z_t^\eps\).
Recall that by the   decomposition in Lemma~\ref{lem:decz}, we have
\[
V_t^\eps=\gamma(X_t^\eps)\left(Z_t^\eps-Y_t^\eps+\sqrt{\eps}\gamma^{-1}(x)v\right).
\]
Substituting this identity into \eqref{eq:zeps}, we obtain
\begin{align}\label{eq:zfast}
\dif Z_t^\eps&=-\frac1\eps\gamma(X_t^\eps)\big(Z_t^\eps-Y_t^\eps\big)\dif t
+\frac1{\sqrt{\eps}}a(\bar X_t)\dif W_t\no\\
&\quad+\frac1{\sqrt{\eps}}\left[B(\bar X_t)-\gamma(X_t^\eps)\gamma^{-1}(x)v\right]\dif t.
\end{align}
Thus, on the fast time scale, \(Z_t^\varepsilon\) is close to the
frozen Ornstein--Uhlenbeck process (\ref{eq:forz1}) with parameters
\((X_t^\varepsilon,Y_t^\varepsilon)\).
The following proposition quantifies this conditional equilibration.

\begin{proposition}\label{prop:zcon}
Assume {\bf(A1)} and {\bf(A2)} hold. Then, for every \(t\in(0,T]\) and \(\varphi\in C_b^{\a}\) with $\alpha\in(0,1)$, we have
\begin{equation*}
\left|\E\varphi(Z_t^\eps)-\E\bar\varphi(X_t^\eps,Y_t^\eps)\right|
\leq C_{T,\varphi}\,\sqrt\eps+C_\varphi\,\mathrm e^{-\kappa t/\eps},
\end{equation*}
where $\bar\varphi$ is defined by \eqref{barphi}, \(C_{T,\varphi}\) is a constant independent of \(\eps\), and \(C_\varphi,\kappa>0\) are independent of \(\eps\) and \(T\).
\end{proposition}

\begin{proof}
Recall that by (\ref{ww}), we have
\begin{equation*}
w(s,z;x,y)=\cP_s^{x,y}\varphi(z)-\bar\varphi(x,y)
=\E\varphi(Z^{x,y}_s(z))-\bar\varphi(x,y).
\end{equation*}
Then \(w\) satisfies the following Cauchy problem:
\begin{equation}\label{eq:cau1}
\left\{\begin{aligned}
&\p_s w(s,z;x,y)-\sL_1(x,y)w(s,z;x,y)=0,\\
&w(0,z;x,y)=\varphi(z)-\bar\varphi(x,y),
\end{aligned} \right.
\end{equation}
where the operator $\sL_1(x,y)$ is defined by
\begin{align*}
\mathscr L_1(x,y)
:=
-\langle\gamma(x)(z-y),\nabla_z\rangle
+\frac12a(x)a^*(x):\nabla_z^2.
\end{align*}
For $s\leq T$, we define
\[
\tilde w(s,z;x,y):=w(T-s,z;x,y).
\]
It then follows that
\begin{align*}
&\tilde w(T,z;x,y)=w(0,z;x,y)=\varphi(z)-\bar\varphi(x,y),\\
&\tilde w(0,z;x,y)=w(T,z;x,y)=\E\varphi(Z^{x,y}_T(z))-\bar\varphi(x,y).
\end{align*}
Since $Z_0^\varepsilon=Y_0^\varepsilon=0$,
$X_0^\varepsilon=x$, and taking \(T=t/\eps\), we have
\begin{align}\label{eq:v1}
\cV(\eps)&:=\left|\E\varphi(Z_t^\eps)-\E\bar\varphi(X_t^\eps,Y_t^\eps)\right|\no\\
&\leq\left|\E\tilde w\left(\frac{t}{\eps},Z_t^\eps;X_t^\eps,Y_t^\eps\right)-\tilde w(0,0;x,0)\right|+\left|\tilde w(0,0;x,0)\right|\no\\
&=:\cV_1(\eps)+\cV_2(\eps).
\end{align}
For the second term, by the estimate \eqref{eq:estou}  we obtain
\begin{equation}\label{eq:v2}
\cV_2(\eps)=\left|w\left(\frac{t}{\eps},0;x,0\right)\right|\leq C_\varphi\,\mathrm e^{-\kappa t/\eps}.
\end{equation}
We next estimate \(\cV_1(\varepsilon)\).
Using It\^{o}'s formula, by (\ref{eq:zfast}) and \eqref{eq:cau1}, we deduce that
\begin{align}\label{eq:v3}
&\E\tilde w\Big(\frac{t}{\eps},Z_t^\eps;X_t^\eps,Y_t^\eps\Big)-\tilde w(0,0;x,0)=\frac{1}{\sqrt{\eps}}\E\left(\int_0^t\Big\langle \nabla_x\tilde w\Big(\frac{s}{\eps},Z_s^\eps;X_s^\eps,Y_s^\eps\Big),V_s^\eps\Big\rangle\dif s\right)\no\\
&\quad+\frac{1}{\sqrt{\eps}}\E\left(\int_0^t\Big\langle \nabla_z\tilde w\Big(\frac{s}{\eps},Z_s^\eps;X_s^\eps,Y_s^\eps\Big),B(\bar X_s)-\g(X^\eps_s)\g^{-1}(x)v\Big\rangle\dif s\right)\no\\
&\quad+\frac{1}{\sqrt{\eps}}\E\left(\int_0^t\Big\langle \nabla_y\tilde w\Big(\frac{s}{\eps},Z_s^\eps;X_s^\eps,Y_s^\eps\Big),B(\bar X_s)-B(X^\eps_s)+h(X^\eps_s,V^\eps_s)\Big\rangle\dif s\right)\no\\
&\quad+\frac{1}{2\eps}\E\left(\int_0^t\nabla^2_{z}\tilde w\Big(\frac{s}{\eps},Z_s^\eps;X_s^\eps,Y_s^\eps\Big):\big[a(\bar X_s)a^*(\bar X_s)-a(X_s^\eps)a^*(X_s^\eps)\big]\dif s\right)\no\\
&\quad+\frac{1}{2\eps}\E\left(\int_0^t\nabla^2_{y}\tilde w\Big(\frac{s}{\eps},Z_s^\eps;X_s^\eps,Y_s^\eps\Big):\left[\big(a(\bar X_s)-a(X^\eps_s)\big)\big(a(\bar X_s)-a(X^\eps_s)\big)^*\right]\dif s\right)\no\\
&\quad+\frac{1}{\eps}\E\left(\int_0^t\nabla^2_{yz}\tilde w\Big(\frac{s}{\eps},Z_s^\eps;X_s^\eps,Y_s^\eps\Big):\left[\big(a(\bar X_s)-a(X^\eps_s)\big)a^*(\bar X_s)\right]\dif s\right)=:\sum_{i=1}^6\sI_{i}(\eps).
\end{align}
Using the derivative estimates \eqref{eq:estou} and the moment estimates \eqref{eq:estv}, \eqref{eq:estz} and \eqref{eq:esty}, we have
\begin{align*}
|\sI_1(\eps)|+|\sI_3(\eps)|&\leq\frac{C_1}{\sqrt{\eps}}\int_0^t \mathrm e^{-\kappa(t-s)/\eps}\E\big[1+|V_s^\eps|^p
+|Y_s^\eps|^p+|Z_s^\eps|^p\big]\dif s\\
&\leq\frac{C_T}{\sqrt{\eps}}\int_0^t\mathrm e^{-\kappa(t-s)/\eps}\dif s
\leq C_T\sqrt{\eps}.
\end{align*}
Similarly, by \eqref{eq:estou-z} we have
\begin{align*}
|\sI_2(\varepsilon)|
\leq
\frac{C_T}
{\sqrt{\varepsilon}}
\int_0^t \left(\frac{(t-s)}{\eps}\wedge1\right)^{-(1-\a)/2}\mathrm e^{-\kappa(t-s)/\eps}\dif s
\leq
C_{T}
\sqrt{\varepsilon}.
\end{align*}
For the term involving \(\nabla_{z}^2\tilde w\), by the boundedness of $a$ and $\nabla a$, it holds
\[
\big|a(\bar X_{s})a^*(\bar X_{s})-a(X^\eps_{s})a^*(X^\eps_s)\big|\leq C_2\big|a(\bar X_{s})-a(X^\eps_{s})\big|\leq C_2\sqrt{\eps}|Z_s^\eps|.
\]
This together with \eqref{eq:estz}, \eqref{eq:esty} and \eqref{eq:estou-zz} yields that
\begin{align*}
\sI_4(\eps)&\leq\frac{C_3}{\sqrt{\eps}}\int_0^t \left(\frac{(t-s)}{\eps}\wedge1\right)^{-1+\a/2}
\mathrm e^{-\kappa(t-s)/\eps}\E\big[1
+|Y_s^\eps|^p+|Z_s^\eps|^p\big]\dif s\\
&\leq\frac{C_T}{\sqrt{\eps}}\int_0^t \left(\frac{(t-s)}{\eps}\wedge1\right)^{-1+\a/2}\mathrm e^{-\kappa(t-s)/\eps}\dif s
\leq C_T\sqrt{\eps}.
\end{align*}
Similarly, we have
\[
\norm{a(\bar X_s)-a(X_s^\eps)}\leq C_4\sqrt{\eps}\norm{Z_s^\eps},
\]
and by \eqref{eq:estou} we deduce that
\begin{align*}
\sI_5(\eps)\leq C_T\int_0^t \mathrm e^{-\kappa(t-s)/\eps}\dif s
\leq C_T\,\eps.
\end{align*}
Finally, for the last term and using the similar arguments as above, we get
\begin{align*}
\sI_6(\eps)\leq\frac{C_T}{\sqrt{\eps}}\int_0^t\mathrm e^{-\kappa(t-s)/\eps}\dif s
\leq C_T\sqrt{\eps}.
\end{align*}
Combining these estimates with \eqref{eq:v3}, we obtain
\[
\cV_1(\eps)=\Big|\E\tilde w\Big(\frac{t}{\eps},Z_t^\eps;X_t^\eps,Y_t^\eps\Big)-\tilde w(0,0;x,0)\Big|\leq C_T\sqrt{\eps},
\]
which together with \eqref{eq:v1} and \eqref{eq:v2} gives
\[
\cV(\eps)\leq C_T\sqrt{\eps}+C_0\,\mathrm e^{-\kappa t/\eps}.
\]
The proof is finished.
\end{proof}

\section{Proof of the main theorem}
\label{sec:proof-main}

We now combine the joint diffusion limit for the slow pair
\((X^\eps_t,Y^\eps_t)\) with the conditional equilibrium
estimate for the fast component to give:

\begin{proof}[{\bf Proof of Theorem~\ref{thm:main}}]
Let \(\varphi\in C_b^\a\) and $
\bar\varphi(x,y)
$ be defined by (\ref{barphi}).
It follows from Lemma~\ref{lem:estou} that
\(
\bar\varphi\in C_b^{2+\beta,2+\beta}.
\)
As a direct result of Proposition~\ref{prop:zcon}, we obtain
\begin{equation*}
\left|\E\varphi(Z_t^\eps)-\E\bar\varphi(X_t^\eps,Y_t^\eps)\right|\leq C_{T,\varphi}\,\sqrt{\eps}+C_\varphi\,\mathrm e^{-\kappa t/\eps}\leq C_{T,\varphi}\,\eps^{\beta/2}+C_\varphi\,\mathrm e^{-\kappa t/\eps}.
\end{equation*}
On the other hand, by Proposition~\ref{prop:slow} we have
\[
\left|\E\bar\varphi(X_t^\eps,Y_t^\eps)-\E\bar\varphi(\bar X_t,\bar Y_t)\right|
\leq C_{T,\varphi}\,\eps^{\beta/2}.
\]
Combining these estimates, we conclude that
\[
\left|\E\varphi(Z_t^\eps)-\E\bar\varphi(\bar X_t,\bar Y_t)\right|
\leq C_{T,\varphi}\,\eps^{\beta/2}+C_\varphi\,\mathrm e^{-\kappa t/\eps},
\]
which proves \eqref{eq:mainest}.

To identify the limiting random variable, enlarge the
probability space, if necessary, by introducing a standard Gaussian
random variable \(\xi\sim\cN(0,\Id_d)\), independent of
\((\bar X_t,\bar Y_t)\), and define
\begin{equation*}
Z_t:=\bar Y_t+N^{1/2}(\bar X_t)\xi.
\end{equation*}
Conditionally on \((\bar X_t,\bar Y_t)\), we have
\[
\Law(Z_t\mid\bar X_t,\bar Y_t)=\cN(\bar Y_t,N(\bar X_t))=\mu^{\bar X_t,\bar Y_t},
\]
where $\mu^{x,y}$ is given by (\ref{muxy}).
As a result,
\begin{equation*}
\E\varphi(Z_t)=\E\left(\int_{\R^d}\varphi(z)\,\mu^{\bar X_t,\bar Y_t}(\dif z)\right)
=\E\bar\varphi(\bar X_t,\bar Y_t),
\end{equation*}
and
\[
\left|\E\varphi(Z_t^\varepsilon)-\E\varphi(Z_t)\right|
\leq C_{T,\varphi}\,\varepsilon^{\beta/2}+C_\varphi\,\mathrm e^{-\kappa t/\varepsilon}.
\]
The convergence for arbitrary bounded continuous test functions, and
hence the weak convergence stated in \eqref{eq:jointcon}, follows by the standard
  approximation argument.
\end{proof}

\begin{remark}[Joint formulation]
The fixed-time convergence stated in Theorem~\ref{thm:main} can also be formulated jointly with the
slow position. More precisely, for every
$\psi\in C_b^{\a,2+\beta}$,
 the same
argument of the proof gives
\begin{align*}
&\left|
\E\psi(Z_t^\varepsilon,X_t^\varepsilon)
-
\E\left(
\int_{\mathbb R^d}
\psi(z,\bar X_t)\,
\mu^{\bar X_t,\bar Y_t}(\mathrm dz)
\right)
\right|\\
&\qquad\leq
C_{T,\psi}\,\varepsilon^{\beta/2}
+
C_\psi\,\mathrm e^{-\kappa t/\varepsilon},
\qquad t\in(0,T].
\end{align*}
Indeed, one applies the frozen Ornstein--Uhlenbeck semigroup (\ref{sem}) to
$z\mapsto\psi(z,x)$, treating $x$ as a parameter, and then applies
the slow pair control to
\[
\bar\psi(x,y)
:=
\int_{\mathbb R^d}\psi(z,x)\,\mu^{x,y}(\mathrm dz).
\]
The additional derivatives in $x$ are controlled by the assumed
regularity of $\psi$ and the same Gaussian-kernel estimates used
above.

Consequently, for every fixed $t>0$, we have
\[
(Z_t^\varepsilon,X_t^\varepsilon)
\Longrightarrow
\bigl(
\bar Y_t+N^{1/2}(\bar X_t)\xi,\bar X_t
\bigr),
\]
where $\xi\sim\mathcal N(0,\mI_d)$ is independent of
$(\bar X_t,\bar Y_t)$.
This formulation retains the dependence between the fluctuation
and the limiting position.
\end{remark}

\vspace{2mm}

\noindent{\bf\large Declarations}

\vspace{1mm}
\noindent{\bf Conflict of interest} The authors declared that they have no conflict of interest to this work.

\noindent{\bf Data availibility} Data sharing not applicable to this article as no datasets were generated or analysed during
the current study.

\end{document}